\documentclass{amsart}

\usepackage{amsmath,amssymb}
\usepackage{amsthm}
\usepackage{amsrefs}
\usepackage{indentfirst}
\usepackage{mathrsfs}
\usepackage{hyperref}
\usepackage{xcolor}
\usepackage[margin=1in]{geometry}
\usepackage{nicefrac}
\usepackage{enumerate}
\usepackage{graphicx}
\usepackage{subfigure}

\DeclareMathOperator*{\argmin}{argmin}

\newlength{\leftstackrelawd}
\newlength{\leftstackrelbwd}
\def\leftstackrel#1#2{\settowidth{\leftstackrelawd}%
	{${{}^{#1}}$}\settowidth{\leftstackrelbwd}{$#2$}%
	\addtolength{\leftstackrelawd}{-\leftstackrelbwd}%
	\leavevmode\ifthenelse{\lengthtest{\leftstackrelawd>0pt}}%
	{\kern-.5\leftstackrelawd}{}\mathrel{\mathop{#2}\limits^{#1}}}

\numberwithin{equation}{section}
\newtheorem{theorem}{Theorem}[section]
\newtheorem{lemma}[theorem]{Lemma}
\newtheorem{remark}[theorem]{Remark}

\newtheorem{assumption}[theorem]{Assumption}
\newtheorem{definition}[theorem]{Definition}

\allowdisplaybreaks[4]

\title[Sobolev Gradient Flows for Rotating Gross-Pitaevskii Energy Functional]{Convergence Analysis of Sobolev Gradient Flows for the Rotating Gross-Pitaevskii Energy Functional}
\author{Chen Zhang}
\address{(CZ) School of Mathematical Sciences, Fudan University, Shanghai, 200433, People's Republic of China. }
\email{24210180137@m.fudan.edu.cn}
\author{Patrick Henning}
\address{(PH) Department of Mathematics, Ruhr-University Bochum, DE-44801 Bochum, Germany. }
\email{patrick.henning@rub.de}
\author{Mahima Yadav}
\address{(MY) Department of Mathematics, Ruhr-University Bochum, DE-44801 Bochum, Germany. }
\email{mahima.yadav@rub.de}
\author{Wenbin Chen}
\address{(WC) School of Mathematical Sciences and Shanghai Key Laboratory for Contemporary Applied Mathematics, Fudan University, Shanghai, 200433, People's Republic of China.}
\email{wbchen@fudan.edu.cn}

\thanks{This paper was submitted for publication on October 8, 2025.}

\begin{document}

\begin{abstract}
	This paper studies the numerical approximation of the ground state of rotating Bose--Einstein condensates, formulated as the minimization of the Gross--Pitaevskii energy functional under a mass conservation constraint. To solve this problem, we consider three Sobolev gradient flow schemes: the $H_0^1$ scheme, the $a_0$ scheme, and the $a_u$ scheme. Convergence of these schemes in the non-rotating case was established by Chen et al.~\cite{chen2024convergence}, and the rotating $a_u$ scheme was analyzed in Henning et al.~\cite{henning2025convergence}. In this work, we prove the global convergence of the $H_0^1$ and $a_0$ schemes in the rotating case, and establish local linear convergence for all three schemes near the ground state. Numerical experiments confirm our theoretical findings.
\end{abstract}
\keywords{Gross--Pitaevskii energy functional, Bose--Einstein condensates, Sobolev gradient flow.}
	
\maketitle

\section{Introduction}

At extremely low temperatures, dilute bosonic gases exhibit a peculiar state of matter known as Bose--Einstein condensate, or BEC for short~\cites{alma991054761079706011,bose1924plancks,einstein2005quantentheorie}. The ground state of BEC, which is the lowest energy state, is a crucial research problem that has received significant attention. Mathematically, the ground state is defined as the minimizer of the Gross--Pitaevskii energy functional on the $L^2$ sphere, where the energy functional takes the form as following~\cites{bao2012mathematical,bao2014mathematical}: for any $u \in \mathcal{M}$,
\begin{equation*}
    E(u) := \frac{1}{2}\int_{\mathcal{D}} \left(|\nabla u|^2 + V|u|^2 + \frac{\beta}{2}|u|^4 - \Omega \overline{u} L_z u\right) dr,
\end{equation*}
The $L^2$ sphere $\mathcal{M}$ is defined as
\begin{equation*}
    \mathcal{M} := \{u\in H_0^1(\mathcal{D}; \mathbb{C}): \|u\|_{L^2(\mathcal{D})}=1\},
\end{equation*}
where the integration domain $\mathcal{D}\subseteq \mathbb{R}^d$ ($d=2,3$) is a bounded region, $\mathbb{C}$ denotes the complex number field, the domain of the functional is $H_0^1(\mathcal{D}; \mathbb{C})$, the function $V\in L^{\infty}(\mathcal{D};\mathbb{R}_{\ge 0})$ represents the potential, the constant $\beta\in \mathbb{R}_{\ge 0}$ describes interparticle repulsion, $\Omega\in \mathbb{R}_{\ge 0}$ is the angular velocity, and $L_z$ denotes the third component of angular momentum component, defined as
\begin{equation*}
    L_z := -\mathrm{i} (x\partial_y - y\partial_x),
\end{equation*}
where the imaginary unit $\mathrm{i}=\sqrt{-1}$, and $\overline{u}$ denotes the complex conjugate of the wave function $u$. The $L^2$ spherical condition can be regarded as a constraint for the mass of the condensates. For fundamental theories and mathematical properties of the ground state, we refer to~\cites{bao2005ground,bao2012mathematical,bao2014mathematical}. A particularly important case is the non-rotating scenario ($\Omega = 0$), where formal analytical solutions can be derived~\cite{trallero2008formal} and the stability of solutions has been studied~\cite{jackson2005stability}.

By considering the constrained Euler--Lagrange equations, the optimization problem for the ground state of BEC can be transformed into the Gross--Pitaevskii eigenvalue equation:
\begin{equation}\label{equation_2_7}
    -\Delta u + Vu + \beta |u|^2 u - \Omega L_z u = \lambda u,
\end{equation}
After spatial discretization, this problem can be solved using eigenvalue methods. For the discretization, one can consider spectral and pseudo-spectral methods~\cites{bao2012mathematical,bao2014mathematical,cances2010numerical,antoine2014robust}, Lagrange finite elements methods~\cites{danaila2010finite,chen2011adaptive,cances2010numerical,henning2025discrete,zhou2003analysis}, mixed FEM~\cite{GHLP25}, discontinuous Galerkin composite FEM~\cite{engstrom2022higher}, spectral element methods~\cite{chen2024fully}, generalized finite element methods~\cites{henning2014two,henning2023optimal}, multigrid methods~\cites{zhang2019efficient,xie2016multigrid}, etc. General  approximation results are found in~\cite{zhou2003analysis} and adaptivity and a posteriori error estimates are e.g. addressed in~\cites{heid2021gradient,dusson2017posteriori,dussonmaday23}.

To solve the discretized problem, various eigenvalue methods have been proposed, such as self-consistent field iterations~\cites{cances2000can,upadhyaya2020density,dion2007ground,cances2000convergence,defranceschi2000scf}, generalized inverse iterations~\cites{jarlebring2014inverse,henning2023dependency,altmann2021j}, two-grid continuation methods~\cite{chien2006two} or Newton-type methods~\cites{xu2021multigrid,wu2017regularized,altmann2024riemannian}. For the latest advancements on this problem, we refer to~\cite{henning2025gross}. 

In addition to these eigenvalue methods, gradient flow approaches have also played a significant role, most notably the $L^2$ gradient flow~\cite{bao2004computing} and various Sobolev gradient flows~\cites{APS22,danaila2010new,danaila2017computation,zhang2022exponential,heid2021gradient,kazemi2010minimizing,chen2024convergence,henning2020sobolev}. In this work, we focus on the projected Sobolev gradient flow method. In short, the gradient flow we consider takes the form~\cite{chen2024convergence}
\begin{equation*}
    \frac{\partial u}{\partial t} = -\nabla_{X}^{\mathcal{R}}E(u) = -\mathcal{P}_{u, X}(\nabla_{X}E(u)),
\end{equation*}
where $\nabla_{X}^{\mathcal{R}}E(u(t))$ is the Riemannian gradient of the functional $E$ with respect to the $X$-inner product, i.e., the projection of the gradient $\nabla_{X}E(u(t))$ onto the tangent space $T_u \mathcal{M}$ under the operator $\mathcal{P}_{u, X}$. Specifically, three gradient flow schemes have received particular attention: the $H^1$ gradient flow scheme~\cite{kazemi2010minimizing}, the $a_0$ gradient flow scheme~\cite{danaila2010new}, and the $a_u$ gradient flow scheme~\cite{henning2020sobolev}.

From an algorithmic perspective, these gradient flow algorithms are rich and highly efficient, and their convergence analysis is critical. Most current convergence results focus on the non-rotating case. Early work by Kazemi and Eckart~\cite{kazemi2010minimizing} demonstrated the global convergence of the continuous $H_0^1$ gradient flow. In contrast, Faou and Jézéquel~\cite{faou2018convergence} subsequently established the local convergence of the discrete $L^2$ gradient flow. Extending these investigations, Henning and Peterseim~\cite{henning2020sobolev} provided a global convergence proof for both continuous and semidiscrete $a_u$ gradient flows. Zhang~\cite{zhang2022exponential} further explored the local convergence of the semidiscrete $a_u$ gradient flow under specific conditions. More recently, significant progress has been made by Chen et al.~\cite{chen2024convergence}, who proved global and local convergence for semidiscrete $H_0^1$ and $a_0$ gradient flows, alongside establishing local convergence for a more general semidiscrete $a_u$ gradient flow. 

Compared to the non-rotating case, the rotating case is more challenging, as the rotation term leads to difficulties such as non-uniqueness of the ground state and the breakdown of linearization analysis. Consequently, the proof technique of~\cite{chen2024convergence} becomes ineffective in the rotating frame. Recently, Henning and Yadav~\cite{henning2025convergence} established global and local convergence results for the semidiscrete $a_u$ gradient flow. However, their auxiliary iteration and fixed-point methods cannot be directly extended to the other two gradient flows, which poses additional challenges for our analysis. In this paper, building on the framework of~\cites{henning2025convergence,henning2025discrete} and combining it with the methods of~\cite{chen2024convergence}, we study the quotient space of $H_0^1$ function spaces and its associated quotient metric to derive convergence results in the rotating case.

A recent and independent preprint by Feng and Tang~\cite{FengTang2025} analyzes precisely the same Sobolev gradient schemes considered here, formulated within a unified preconditioned Riemannian-gradient framework. They establish global convergence of the corresponding gradient flow under general conditions, and, under an additional Morse--Bott assumption, prove a Polyak--\L{}ojasiewicz inequality leading to explicit local linear convergence rates and an expression for an optimal preconditioner. Their results complement the present work. Although the underlying algorithms coincide, the analytical perspectives differ: the present paper provides detailed, scheme-specific convergence proofs for these classical Sobolev gradient flows, including global convergence for the $H_0^1$ and $a_0$ schemes and local linear convergence in a quotient metric for all three schemes, based on direct energy estimates and quotient-space arguments. In contrast to the Polyak--\L{}ojasiewicz approach, our analysis follows the energy-method techniques in the spirit of~\cite{chen2024convergence} and establishes convergence properties by direct control of the energy and its second variation, which allows us to handle the phase invariance of the Gross--Pitaevskii functional in a constructive way. Thus, while~\cite{FengTang2025} offers a unified and quantitatively sharp Riemannian framework highlighting optimal preconditioning, our analysis provides a complementary, constructive viewpoint at the level of concrete Sobolev gradient flows commonly used in computations.

The remainder of this paper is structured as follows. Section \ref{section_2} reviews the three Sobolev gradient flow schemes. Section \ref{section_3} presents the main convergence results. Section \ref{section_4} establishes global convergence, while Section \ref{section_5} addresses local convergence. Finally, Section \ref{section_6} reports the numerical experiments.

\section{Sobolev Gradient Flow}\label{section_2}

Consider the constrained optimization problem for the Gross--Pitaevskii energy functional with a rotating term:
\begin{equation}\label{equation_2_6}
    \min_{u \in \mathcal{M}} E(u) := \int_\mathcal{D} \left( \frac{1}{2}|\nabla u|^2 + \frac{1}{2}V|u|^2 + \frac{\beta}{4}|u|^4 - \frac{\Omega}{2} \overline{u} L_z u \right) dr,
\end{equation}
where the integration domain $\mathcal{D}\subseteq \mathbb{R}^d$ $(d=2,3)$ is bounded, with its bound $M=\sup\{|r|:r\in \mathcal{D}\}$. 

For a point in $\mathcal{D}$, we further denote $r=(x,y)$ for $d=2$ and $r=(x,y,z)$ for $d=3$. The constraint manifold is $\mathcal{M}=\{u\in H_0^1(\mathcal{D};\mathbb{C}):\|u\|_{L^2(\mathcal{D})}=1\}$, $V\in L^{\infty}(\mathcal{D}; \mathbb{R}_{\ge 0})$ is the trapping potential, with $V_{\max}:=\|V\|_{L^{\infty}}$; $\beta>0$ describes the strength of particle repulsion, $\Omega\ge 0$ is the angular velocity of rotation, and $L_z=-\mathrm{i}(x\partial_y-y\partial_x)$ is the $z$-component of the angular momentum. To guarantee the existence of a ground state, the following assumption is required~\cite{henning2025convergence}:
\begin{assumption}\label{assumption_2_4}
    There exists a positive constant $K$ such that the potential $V$ and the angular velocity $\Omega$ satisfy
    \begin{equation*}
        V(r) \ge \frac{1+K}{4}\Omega^2 (x^2 + y^2)\quad\quad \forall r\in \mathcal{D}. 
    \end{equation*}
\end{assumption}
This assumption has a clear physical interpretation: it guarantees that centrifugal forces do not outweigh the strength of the trapping potential.

For brevity, the domain $\mathcal{D}$ is omitted when writing the inner product $(\cdot,\cdot)_{X(\mathcal{D})}$ and norm $\|\cdot\|_{X(\mathcal{D})}$. $H_0^1(\mathcal{D};\mathbb{C})$ is abbreviated as $H_0^1(\mathcal{D})$, and the same applies to other complex-valued function spaces. In the paper, we need the following Sobolev embedding. 

\begin{lemma}\label{lemma_2_1}
    For $d=2, 3$, there exist constants $C_1$, $C_2$, $C_3$, and $C_4$ depending only on $\mathcal{D}$ and $d$, such that
    \begin{align*}
        &\|u\|_{L^4} \le C_1 \|u\|_{H_0^1} \quad\quad\forall u\in H_0^1(\mathcal{D}),\\
        &\|u\|_{H^{-1}} \le C_2 \|u\|_{L^{4/3}} \quad\quad\forall u\in L^{4/3}(\mathcal{D}),\\
        &\|u\|_{L^2} \le C_3 \|u\|_{H_0^1} \quad\quad\forall u\in H_0^1(\mathcal{D}), \\
        &\|u\|_{L^6} \le C_4 \|u\|_{H_0^1} \quad\quad\forall u\in H_0^1(\mathcal{D}).
    \end{align*}
    The third inequality is also known as Poincaré's inequality. 
\end{lemma}

For a constrained optimization problem, a standard approach, as seen in~\cite{henning2025convergence}, introduces a Lagrange multiplier $\lambda$ and rewrites the constraint condition as $\frac{1}{2}(\|u\|_{L^2}^2-1)=0$, leading to the Euler--Lagrange equation
\begin{equation}\label{equation_2_1}
    \langle E'(u),v \rangle = \lambda \Re(u,v)_{L^2}\quad\quad\forall v\in H_0^1(\mathcal{D}),
\end{equation}
where $\langle\cdot,\cdot\rangle$ denotes the duality pairing between the space $H_0^1(\mathcal{D})$ and its dual space $H^{-1}(\mathcal{D})$, $\Re$ denotes the real part, and $E'(u)\in H^{-1}(\mathcal{D})$ is the real-Fréchet derivative of the functional $E(u)$:
\begin{equation}\label{equation_2_2}
    \langle E'(u),v \rangle = \Re\left(\nabla u, \nabla v\right)_{L^2} +  \Re\left( Vu + \beta|u|^2 u - \Omega L_z u, v\right)_{L^2} \quad\quad\forall v\in H_0^1(\mathcal{D}).
\end{equation}
Substituting equation \eqref{equation_2_2} into equation \eqref{equation_2_1} yields the Gross--Pitaevskii eigenvalue problem \eqref{equation_2_7}.

In the following, we denote $u^*:=\argmin\limits_{u \in \mathcal{M}} E(u)$, and its corresponding eigenvalue as $\lambda$. If $(\cdot,\cdot)_X$ defines an inner product on $H_0^1(\mathcal{D})$, the Riesz representation theorem guarantees the existence of a unique function $\nabla_X E(u) \in H_0^1(\mathcal{D})$ such that
\begin{equation}\label{equation_2_5}
    \Re (\nabla_X E(u), h)_X = \langle E'(u), h \rangle \quad\quad \forall h \in H_0^1(\mathcal{D}).
\end{equation}
The function $\nabla_X E(u)$ satisfying equation \eqref{equation_2_5} is the Sobolev gradient of the functional $E(u)$ with respect to the $X$-metric. To ensure that the iterative process remains on the manifold $\mathcal{M}$, we use the projected Sobolev gradient on the tangent space. At $u\in\mathcal{M}$, the tangent space of the manifold is given by
\begin{equation*}
    T_u \mathcal{M} = \{ v \in H_0^1(\mathcal{D}; \mathbb{C}) : \Re(u, v)_{L^2} = 0 \}, 
\end{equation*}
and the $X$-orthogonal projection on $T_u \mathcal{M}$ is
\begin{equation*}
    \mathcal{P}_{u,X} w = w - \frac{\Re(w, \mathcal{G}_X u)_X}{\|\mathcal{G}_X u\|_X^2} \mathcal{G}_X u, 
\end{equation*}
where the operator $\mathcal{G}_X\in L(H^{-1}(\mathcal{D}), H_0^1(\mathcal{D}))$ satisfies
\begin{equation*}
    (\mathcal{G}_X(w), h)_X = (w, h)_{L^2} \quad\quad \forall w,h \in H_0^1(\mathcal{D}). 
\end{equation*}
Accordingly, the projected Sobolev gradient of $\nabla_X E(u)$ (also known as Riemmanian gradient) is
\begin{equation*}
    \nabla^{\mathcal{R}}_X E(u) = \mathcal{P}_{u,X}(\nabla_X E(u)),    
\end{equation*}
and the corresponding (continuous) Sobolev gradient flow reads
\begin{equation*}
    \frac{\partial u}{\partial t} = -\nabla^{\mathcal{R}}_X E(u),
\end{equation*}
This paper mainly focuses on the forward Euler scheme for the continuous Sobolev gradient flow, which is
\begin{equation*}
    u_{n+1} = R\left(u_n - \alpha_n \nabla^{\mathcal{R}}_X E(u_n)\right), 
\end{equation*}
where the step size $\alpha_n$ is bounded, satisfying $0<\alpha_{\min}\le\alpha_n\le\alpha_{\max}$, and $R$ is the retraction map
\begin{equation*}
    R(v) := \frac{v}{\|v\|_{L^2}}\quad\quad\forall v\in H^1_0(\mathcal{D})\setminus \{0\}, 
\end{equation*}
ensuring that each iteration step remains on the constraint manifold $\mathcal{M}$. 

Next, we introduce the three main inner products $X$ involved in this paper, along with their corresponding gradients $\nabla_X E(u)$ and projected gradients $\nabla^{\mathcal{R}}_X E(u)$. With these projected gradients, we can construct the respective $X$ schemes. 

\begin{enumerate}
    \item $H_0^1$ scheme~\cite{kazemi2010minimizing}. Taking the inner product $(\cdot,\cdot)_X=(\cdot,\cdot)_{H_0^1}$, defined as
        \begin{equation*}
            (v, w)_{H_0^1} := (\nabla v, \nabla w)_{L^2} = \int_\mathcal{D} \nabla \overline{v} \cdot \nabla w dr\quad\quad\forall v,w\in H_0^1(\mathcal{D}; \mathbb{C}). 
        \end{equation*}
        Then $\mathcal{G}_{H_0^1}=(-\Delta)^{-1}$. The corresponding gradient and projected gradient are
        \begin{align*}
            \nabla_{H_0^1} E(u) &= u + \mathcal{G}_{H_0^1}(Vu + \beta|u|^2 u - \Omega L_z u), \\
            \nabla^{\mathcal{R}}_{H_0^1} E(u) &= \nabla_{H_0^1} E(u) - \frac{1+\Re(u, \mathcal{G}_{H_0^1}(Vu+\beta|u|^2 u - \Omega L_z u))_{L^2}}{\|\mathcal{G}_{H_0^1} u\|_{H_0^1}^2} \mathcal{G}_{H_0^1} u. 
        \end{align*}
    \item $a_0$ scheme~\cite{danaila2010new}. Taking the inner product $(\cdot,\cdot)_X=(\cdot,\cdot)_{a_0}$, defined as
        \begin{equation*}
            (v, w)_{a_0} := \int_\mathcal{D} (\nabla \overline{v} \cdot \nabla w + V \overline{v} w - \Omega \overline{v} L_z w) dr. 
        \end{equation*}
        Then $\mathcal{G}_{a_0} = (-\Delta+V-\Omega L_z)^{-1}$. The corresponding gradient and projected gradient are
        \begin{equation*}
            \nabla_{a_0} E(u) = u + \mathcal{G}_{a_0}(\beta|u|^2 u), \quad\quad \nabla^{\mathcal{R}}_{a_0} E(u) = \nabla_{a_0} E(u) - \frac{1+\Re(u, \mathcal{G}_{a_0}(\beta|u|^2 u))_{L^2}}{\|\mathcal{G}_{a_0} u\|_{a_0}^2} \mathcal{G}_{a_0} u. 
        \end{equation*}
    \item $a_u$ scheme~\cite{henning2020sobolev}. Taking the inner product $(\cdot,\cdot)_X=(\cdot,\cdot)_{a_u}$, defined as
        \begin{equation*}
            (v, w)_{a_u} = \int_\mathcal{D} (\nabla \overline{v} \cdot \nabla w + V \overline{v} w + \beta |u|^2 \overline{v} w - \Omega \overline{v} L_z w) dr.  
        \end{equation*}
        Then $\mathcal{G}_{a_u} = (-\Delta+V+\beta|u|^2-\Omega L_z)^{-1}$. The corresponding gradient and projected gradient are
        \begin{equation*}
            \nabla_{a_u} E(u) =  u, \quad\quad \nabla^{\mathcal{R}}_{a_u} E(u) = u - \frac{\mathcal{G}_{a_u} u}{\|\mathcal{G}_{a_u} u\|_{a_u}^2}. 
        \end{equation*} 
\end{enumerate}

Although these three inner products have been defined above, the positive definiteness of $(\cdot,\cdot)_{a_0}$ and $(\cdot,\cdot)_{a_u}$ as inner products formally requires proof. In fact, the three norms are equivalent, so the above positive definiteness is natural. 

\begin{lemma}[norm equivalence]\label{lemma_2_2}
    For a fixed function $u \in H_0^1(\mathcal{D})$, the $H_0^1$-norm, the $a_0$-norm, and the $a_u$-norm are equivalent. Specifically, for any $v\in H_0^1(\mathcal{D})$, the following inequalities hold:
    \begin{align*}
        &\sqrt{\frac{K}{1+K}}\|v\|_{H_0^1} \le \|v\|_{a_0} \le \left(1 + \Omega M C_3 + C_3^2\|V\|_{L^\infty}\right)^{\frac{1}{2}} \|v\|_{H_0^1}, \\
        &\|v\|_{a_0} \le \|v\|_{a_u} \le \left(1 + \frac{1+K}{K}\beta C_1^4 \|u\|_{H_0^1}^2\right)^{\frac{1}{2}} \|v\|_{a_0}. 
    \end{align*}
\end{lemma}

Lemma \ref{lemma_2_2} can be obtained using Cauchy--Schwarz inequality and Lemma \ref{lemma_2_1}. 

\section{Main Results}\label{section_3}

In this section, we present the main results of this paper, which generalize the convergence results of~\cite{chen2024convergence} to the rotating case. First, we establish global convergence for three schemes. Since the global convergence result for the $a_u$ scheme has been proven in~\cite{henning2025convergence}, here we only provide the results for the $H_0^1$ and $a_0$ schemes.

\begin{theorem}[energy dissipation for the $H_0^1$ scheme]\label{theorem_3_1}
    Let $u_0 \in \mathcal{M}$ and $\{u_n\}_{n=0}^\infty$ be the iteration sequence generated by the $H_0^1$ scheme
    \begin{equation*}
        u_{n+1} = R(u_n - \alpha_n \nabla^{\mathcal{R}}_{H_0^1} E(u_n)). 
    \end{equation*}
    Then there exist positive constants $C_u$, $C_g$, and $C_{\alpha}$, with $C_{\alpha} \le 1$, depending only on $\mathcal{D}$, $d$, $\beta$, $V$, $\Omega$, and $\|u_0\|_{H_0^1}$, such that for any sequence $\{\alpha_n\}_{n=0}^\infty$ satisfying $0 < \alpha_{\min} \le \alpha_n \le \alpha_{\max} \le C_\alpha$, the following properties hold:
    \begin{enumerate}[(i)]
        \item $\|u_n\|_{H_0^1} \le C_u$,
        \item $\|\nabla_{H_0^1}^{\mathcal{R}} E(u_n)\|_{H_0^1} \le \|\nabla_{H_0^1} E(u_n)\|_{H_0^1} \le C_g$, 
        \item $E(u_n) - E(u_{n+1}) \ge \frac{\alpha_{\min}}{2} \|\nabla^{\mathcal{R}}_{H_0^1} E(u_n)\|_{H_0^1}^2$. 
    \end{enumerate}
\end{theorem}

\begin{theorem}[energy dissipation for the $a_0$ scheme]\label{theorem_3_2}
    Let $u_0 \in \mathcal{M}$ and $\{u_n\}_{n=0}^\infty$ be the iteration sequence generated by the $a_0$ scheme
    \begin{equation*}
        u_{n+1} = R(u_n - \alpha_n \nabla^{\mathcal{R}}_{a_0} E(u_n)).
    \end{equation*}
    Then there exist positive constants $C_u$, $C_g$, and $C_{\alpha}$, with $C_{\alpha} \le 1$, depending only on $\mathcal{D}$, $d$, $\beta$, $V$, $\Omega$, and $\|u_0\|_{H_0^1}$, such that for any sequence $\{\alpha_n\}_{n=0}^\infty$ satisfying $0 < \alpha_{\min} \le \alpha_n \le \alpha_{\max} \le C_\alpha$, the following properties hold:
    \begin{enumerate}[(i)]
        \item $\|u_n\|_{a_0} \le C_u$,
        \item $\|\nabla_{H_0^1}^{\mathcal{R}} E(u_n)\|_{a_0} \le \|\nabla_{H_0^1} E(u_n)\|_{a_0} \le C_g$, 
        \item $E(u_n) - E(u_{n+1}) \ge \frac{\alpha_{\min}}{2} \|\nabla^{\mathcal{R}}_{a_0} E(u_n)\|_{a_0}^2$. 
    \end{enumerate}
\end{theorem}

\begin{remark}
    Although the constants $C_u$, $C_g$, and $C_{\alpha}$ are given in both Theorem \ref{theorem_3_1} and Theorem \ref{theorem_3_2}, these constants are actually different for distinct schemes.
\end{remark}

\begin{theorem}[global convergence for the $H_0^1$ scheme]\label{theorem_3_3}
    For the iteration sequence $\{u_n\}_{n=0}^\infty$ generated by the $H_0^1$ scheme, any weak limit point $u^*$ is a stationary point of the optimization problem \eqref{equation_2_6}, and $\{u_n\}_{n=0}^\infty$ has a subsequence that converges strongly to $u^*$ in $H_0^1$ norm.
\end{theorem}

\begin{theorem}[global convergence for the $a_0$ scheme]\label{theorem_3_4}
    For the iteration sequence $\{u_n\}_{n=0}^\infty$ generated by the $a_0$ scheme, any weak limit point $u^*$ is a stationary point of the optimization problem \eqref{equation_2_6}, and $\{u_n\}_{n=0}^\infty$ has a subsequence that converges strongly to $u^*$ in $a_0$ norm.
\end{theorem}

Next, we present the local convergence results for the three schemes, showing that all three converge linearly locally. Unlike the non-rotating case discussed in~\cite{chen2024convergence}, the ground state of the rotating Bose-Einstein energy functional is not unique, even in its neighborhood, because $E(u)=E(\exp(\mathrm{i}\omega)u)$ holds for any $u\in H_0^1(\mathcal{D})$ and $\omega\in [-\pi, \pi)$. However, as $|u|=|\exp(\mathrm{i}\omega)u|^2$, these states share the same density and are hence physically equivalent. Therefore, instead of pursuing uniqueness of the ground state, we assume that the ground state satisfies (locally) quasi-unique property. Simply put, we assume that in a small neighborhood of a ground state $u^*$, any other ground state differs from $u^*$ only by a phase factor. The rigorous definition of this property is as follows, and for more details, readers can refer to~\cites{henning2025discrete,henning2025convergence}.

\begin{definition}\label{definition_3_1}
    A ground state $u^*$ is called a (locally) quasi-unique ground state if for any smooth curve $\gamma: (-1, 1) \to \mathcal{M}$ satisfying $\gamma(0)=u^*$, $\gamma'(0) \in T_{u^*}\mathcal{M} \cap T_{\mathrm{i}u^*}\mathcal{M} \setminus \{0\}$, it holds
    \begin{equation*}
        \frac{d^2}{dt^2} E(\gamma(t)) \Big|_{t=0} > 0.
    \end{equation*}
\end{definition}
In the proof of local convergence, we need the following assumption.
\begin{assumption}[(locally) quasi-uniqueness assumption]\label{assumption_2_1}
    The ground state $u^*$ is a (locally) quasi-unique ground state.
\end{assumption}

At the same time, due to the lack of uniqueness of the ground state, the iteration sequence $\{u_n\}_{n=0}^\infty$ may not even converge to the desired ground state $u^*$. Therefore, we can no longer measure the convergence of the iteration sequence by conventional norms. To address this problem, we consider the equivalence class
\begin{equation*}
    [u] := \{\exp(\mathrm{i}\omega)u : \omega\in [-\pi, \pi)\},
\end{equation*}
which naturally leads to the quotient space and its quotient metric
\begin{equation*}
    H_0^1(\mathcal{D})/\sim := \{[u] : u\in H_0^1(\mathcal{D})\},\quad\quad d\left([u],[v]) := \inf(d(\tilde{u},\tilde{v}) : \tilde{u}\in [u], \tilde{v}\in [v]\right).
\end{equation*}
In particular, by choosing the original metric $d(\cdot, \cdot)$ to be $L^2$ and $H_0^1$ norms, we derive two quotient metrics required in this paper:
\begin{equation*}
    \rho_1(u, v) := \min_{\omega \in [-\pi,\pi)} \|u-e^{\mathrm{i}\omega}v\|_{L^2}, \quad\quad \rho_2(u, v) := \min_{\omega \in [-\pi,\pi)} \|u-e^{\mathrm{i}\omega}v\|_{H_0^1}.
\end{equation*}
For simplicity, for a (locally) quasi-unique ground state $u^*$, we denote
\begin{equation}\label{equation_3_1}
    \rho_1(u) := \rho_1(u, u^*), \quad\quad \rho_2(u) := \rho_2(u, u^*).
\end{equation}

In addition to the lack of uniqueness of the ground state, another difficulty arises from the absence of the following result: in the non-rotating case, if $\lambda$ is the eigenvalue corresponding to the ground state $u^*$ in the Gross--Pitaevskii eigenvalue equation \eqref{equation_2_7}, then $\lambda$ is also the smallest eigenvalue of its linearized eigenvalue problem~\cite{zhang2022exponential}. Specifically, if $\lambda$ is the eigenvalue of the eigenvalue problem
\begin{equation*}
    -\Delta u^* + Vu^* + \beta |u^*|^2 u^* = \lambda u^*,
\end{equation*}
corresponding to the ground state $u^*$ of the Gross-Pitaevskii energy functional, then $\lambda$ is also the smallest eigenvalue of the eigenvalue problem
\begin{equation*}
    -\Delta u + Vu + \beta |u^*|^2 u = \lambda u.
\end{equation*}
This property is used throughout the local convergence proof in~\cite{chen2024convergence}. However, this property does not hold for problems with a rotation term, and a counterexample is given in~\cite{henning2025convergence}. To address this issue, we choose to use $E''(u^*)$ to complete the convergence proof. For a (locally) quasi-unique ground state $u^*$, $E''(u^*)$ has the following important properties:

\begin{lemma}\label{lemma_3_1}
    Let $u^*$ be a (locally) quasi-unique ground state, and let $0 < \lambda_1 \le \lambda_2 \le \cdots$ be the eigenvalues of $E''(u^*)|_{T_{u^*}\mathcal{M}}$. Then $\lambda_2 > \lambda_1$, $\lambda_1 = \lambda$, and the corresponding eigenfunction for $\lambda_1$ is $\mathrm{i}u^*$. Furthermore, for any $v \in T_{u^*}\mathcal{M} \cap T_{\mathrm{i}u^*}\mathcal{M}$, it holds
    \begin{equation*}
        \langle E''(u^*)v, v \rangle - \lambda \|v\|_{L^2}^2 \ge \frac{K(\lambda_2-\lambda_1)}{2(1+K)\lambda_2} \|v\|_{H_0^1}^2.
    \end{equation*}
\end{lemma}

The proof follows as in~\cite{henning2025convergence}*{Lemma 2.3} by exploiting that $\min \{ 1 , (\lambda_2 - \lambda_1) / \lambda_1 \} > (\lambda_2 - \lambda_1) / \lambda_2$. Finally, we present the local convergence theorems. 

\begin{theorem}[local linear convergence for the $H_0^1$ scheme]\label{theorem_3_5}
    If $\rho_2(u_0)$ and $\alpha_{\max}$ are sufficiently small, the sequence $\{u_n\}_{n=1}^{\infty}$ generated by the $H_0^1$ scheme converges linearly to the ground state in the sense of the quotient metric $\rho_2$.
\end{theorem}

\begin{theorem}[local linear convergence for the $a_0$ scheme]\label{theorem_3_6}
    If $\rho_2(u_0)$ and $\alpha_{\max}$ are sufficiently small, the sequence $\{u_n\}_{n=1}^{\infty}$ generated by the $a_0$ scheme converges linearly to the ground state in the sense of the quotient metric $\rho_2$.
\end{theorem}

\begin{theorem}[local linear convergence for the $a_u$ scheme]\label{theorem_3_7}
    If $\rho_2(u_0)$ and $\alpha_{\max}$ are sufficiently small, the sequence $\{u_n\}_{n=1}^{\infty}$ generated by the $a_u$ scheme converges linearly to the ground state in the sense of the quotient metric $\rho_2$.
\end{theorem}

\begin{remark}
    (1) Theorem \ref{theorem_3_5}-\ref{theorem_3_7} simply assume that $\rho_2(u_0)$ and $\alpha_{\max}$ are sufficiently small. In fact, we can provide upper bounds for $\rho_2(u_0)$ and $\alpha_{\max}$, which will be given in the proofs. \\
    (2) Ostrowski's theorem is used in~\cite{henning2025convergence} to prove the local linear convergence for the $a_u$ scheme with constant step size. Theorem \ref{theorem_3_7} extends this to the variable step size case, providing a theoretical foundation for adaptive step size algorithms.
\end{remark}

\section{Proof of Global Convergence}\label{section_4}

In this section, we prove Theorem \ref{theorem_3_1} and Theorem \ref{theorem_3_3} using the method developed in~\cite{chen2024convergence}. Theorems \ref{theorem_3_2} and \ref{theorem_3_4} can be proven similarly. To establish global convergence, we require the following three lemmas. 

\begin{lemma}\label{lemma_4_1}
    For any $u \in L^2(\mathcal{D})$, the following statements hold:
    \begin{equation*}
        \|\mathcal{G}_{H_0^1} u\|_{H_0^1} \le \|u\|_{H^{-1}}, \quad\quad \|\mathcal{G}_{H_0^1} u\|_{H_0^1} \le C_3 \|u\|_{L^2}.
    \end{equation*}
\end{lemma}
The proof of Lemma \ref{lemma_4_1} can be referred to~\cite{chen2024convergence}. 

\begin{lemma}\label{lemma_4_2}
    For any $u \in \mathcal{M}$, it holds that
    \begin{equation*}
        \left\|\nabla^{\mathcal{R}}_{H_0^1} E(u)\right\|_{H_0^1} \le \|\nabla_{H_0^1} E(u)\|_{H_0^1} \le \beta C_1^3 C_2 \|u\|_{H_0^1}^3 + (1 + \Omega M C_3) \|u\|_{H_0^1} + C_3 V_{\max}.
    \end{equation*}
\end{lemma}

\begin{proof}
    Similar to~\cite{chen2024convergence}*{Lemma 4.2}, we can obtain the estimation for terms other than the rotation term:
    \begin{equation*}
        \left\|\nabla^{\mathcal{R}}_{H_0^1} E(u)\right\|_{H_0^1} \le \|\nabla_{H_0^1} E(u)\|_{H_0^1} \le \beta C_1^3 C_2 \|u\|_{H_0^1}^3 + \|u\|_{H_0^1} + C_3 V_{\max} + \Omega \|\mathcal{G}_{H_0^1}(L_z u)\|_{H_0^1}.
    \end{equation*}
    For the rotation term, by Lemma \ref{lemma_4_1}, we have
    \begin{equation*}
        \|\mathcal{G}_{H_0^1}(L_z u)\|_{H_0^1} \le C_3\|L_z u\|_{L^2} \le MC_3\|u\|_{H_0^1}. 
    \end{equation*}
    Combining these two inequalities completes the proof of the lemma. 
\end{proof}

\begin{lemma}\label{lemma_4_4}
    For any $u, v \in H_0^1(\mathcal{D})$, it holds that
    \begin{align*}
        &|E(u+v) - E(u) - \Re(\nabla_{H_0^1} E(u), v)_{H_0^1}| \\
        \le &\frac{1}{2}(1+\Omega M C_3 + V_{\max} C_3^2) \|v\|_{H_0^1}^2 + \frac{3\beta}{2} C_1^4 \|u\|_{H_0^1}^2 \|v\|_{H_0^1}^2 + \beta C_1^4\|u\|_{H_0^1}\|v\|_{H_0^1}^3 + \frac{\beta}{4} C_1^4 \|v\|_{H_0^1}^4.
    \end{align*}
\end{lemma}

\begin{proof}
    By direct calculation, we have
    \begin{align*}
        &E(u+v) - E(u) \\
        =& \Re(\nabla_{H_0^1} E(u), v)_{H_0^1} + \frac{1}{2}\int_{\mathcal{D}}(|\nabla v|^2+V|v|^2)dr - \frac{\Omega}{2}\int_{\mathcal{D}}\overline{v}L_z v dr \\
        &\quad + \int_{\mathcal{D}}\beta \left(\Re^2(\overline{u}v)+\frac{1}{2}|u|^2|v|^2 + |v|^2\Re(\overline{u}v)+\frac{1}{4}|v|^4\right)dr
    \end{align*}
    Using the triangle inequality, Hölder's inequality, Young's inequality, and Lemma \ref{lemma_2_1}, we can obtain
    \begin{align*}
        &\left|E(u+v)-E(u)-\Re(\nabla_{H_0^1}E(u),v)\right|\\
        \le& \frac{1}{2}\int_{\mathcal{D}}(|\nabla v|^2+V|v|^2)dr - \frac{\Omega}{2}\left|\int_{\mathcal{D}}\overline{v}L_z v dr\right| \\
        &\quad + \int_{\mathcal{D}}\beta \left(|u|^2|v|^2+\frac{1}{2}|u|^2|v|^2 + |v|^2|uv|^2+\frac{1}{4}|v|^4\right)dr \\
        \le& \frac{1}{2}(1+V_{\max}C_3^2)\|v\|_{H_0^1}^2 + \frac{\Omega}{4}M(C_3^{-1}\|v\|_{L^2}^2+C_3\|v\|_{H_0^1}^2) \\
        &\quad + \frac{3\beta}{2}\|u\|_{L^4}^2\|v\|_{L^4}^2 + \beta\|u\|_{L^4}\|v\|_{L^4}^3 + \frac{\beta}{4}\|v\|_{L^4}^4 \\
        \le& \frac{1}{2}(1+\Omega M C_3 + V_{\max} C_3^2) \|v\|_{H_0^1}^2 + \frac{3\beta}{2} C_1^4 \|u\|_{H_0^1}^2 \|v\|_{H_0^1}^2 \\
        &\quad + \beta C_1^4\|u\|_{H_0^1}\|v\|_{H_0^1}^3 + \frac{\beta}{4} C_1^4 \|v\|_{H_0^1}^4.
    \end{align*}
    Here, the first inequality uses the triangle inequality, the second inequality applies Lemma \ref{lemma_2_1} to the first term, Young's inequality to the second term, and Hölder's inequality to the remaining terms, and the third inequality uses Lemma \ref{lemma_2_1} to bound all $L^2$ norms by $H_0^1$ norms. 
\end{proof}

We now proceed to prove Theorem \ref{theorem_3_1}, which generalized~\cite{chen2024convergence}*{Theorem 3.1} in the non-rotating case.

\begin{proof}[Proof of Theorem \ref{theorem_3_1}]
    We prove these statements by induction. For $n=0$, (i) holds if $C_u \ge \|u_0\|_{H_0^1}$. Assume that (i) holds for $k\le n$, and (ii) and (iii) hold for $k\le n-1$. We aim to show that (i) holds for $n+1$, and (ii) and (iii) hold for $n$.

    First, since (i) holds for $n$, (ii) holds for $n$ by Lemma \ref{lemma_4_2}. Second, let $g_n = \nabla^{\mathcal{R}}_{H_0^1} E(u_n)$, $\tilde{u}_n = u_n - \alpha_n g_n$, and $R_n=R(\tilde{u}_n)-\tilde{u}_n$. Similar to~\cite{chen2024convergence}*{Lemma 4.3}, we have
    \begin{equation}\label{equation_4_3}
        \begin{aligned}
            \|R_n\|_{H_0^1} &\le \frac{\alpha_n^2}{2} \|g_n\|_{L^2}^2 \|u_n-\alpha_n g_n\|_{H_0^1} \le \frac{\alpha_n^2}{2} (C_u + \alpha_n C_g)C_3^2\|g_n\|_{H_0^1}^2\\
            &\le \frac{\alpha_n^2}{2}(C_u + \alpha_n C_g) C_3^2 C_g^2 \le \frac{1}{2}(C_u + C_g)C_3^2 C_g^2.
        \end{aligned}
    \end{equation}
    On the other hand, similar to the proof of Lemma \ref{lemma_4_2}, we have
    \begin{align*}
        \|\nabla_{H_0^1} E(\tilde{u}_n)\|_{H_0^1} &\le \beta C_1^3 C_2 \|\tilde{u}_n\|_{H_0^1}^3 + (1+\Omega M C_3)\|\tilde{u}_n\|_{H_0^1} + C_3 V_{\max} \|\tilde{u}_n\|_{L^2}\\
        &\le \beta C_1^3 C_2 (C_u+C_g)^3 + (1+\Omega M C_3+C_3^2 V_{\max})(C_u+C_g).
    \end{align*}
    Therefore, by Lemma \ref{lemma_4_4}, we obtain
    \begin{equation}\label{equation_4_1}
        \begin{aligned}
            \left|E(\tilde{u}_n)-E(\tilde{u}_n+R_n)\right| \le& \left|\Re(\nabla_{H_0^1} E(\tilde{u}_n), R_n)_{H_0^1}\right| + \frac{1}{2}(1+\Omega M C_3 + V_{\max} C_3^2) \|R_n\|_{H_0^1}^2\\
            &\quad + \frac{3\beta}{2} C_1^4 \|\tilde{u}_n\|_{H_0^1}^2 \|R_n\|_{H_0^1}^2 + \beta C_1^4\|\tilde{u}_n\|_{H_0^1}\|R_n\|_{H_0^1}^3 + \frac{\beta}{4} C_1^4 \|R_n\|_{H_0^1}^4\\
            \le&  \|R_n\|_{H_0^1}\Big[\|\nabla_{H_0^1} E(\tilde{u}_n)\|_{H_0^1} + \frac{1}{2}(1+\Omega M C_3 + V_{\max} C_3^2) \|R_n\|_{H_0^1}\\
            &\quad + \frac{3\beta}{2} C_1^4 \|\tilde{u}_n\|_{H_0^1}^2 \|R_n\|_{H_0^1} + \beta C_1^4\|\tilde{u}_n\|_{H_0^1}\|R_n\|_{H_0^1}^2 + \frac{\beta}{4} C_1^4 \|R_n\|_{H_0^1}^3\Big]\\
            \le& \alpha_n^2 C_R \|g_n\|_{H_0^1}^2, 
        \end{aligned} 
    \end{equation}
    and
    \begin{equation}\label{equation_4_2}
        \begin{aligned}
            &|E(u_n-\alpha_n g_n) - E(u_n) - \Re(\nabla_{H_0^1} E(u_n), -\alpha_n g_n)_{H_0^1}| \\
            \le &\frac{1}{2}(1+\Omega M C_3 + V_{\max} C_3^2) \|\alpha_n g_n\|_{H_0^1}^2 + \frac{3\beta}{2} C_1^4 \|u_n\|_{H_0^1}^2 \|\alpha_n g_n\|_{H_0^1}^2 \\
            &+ \beta C_1^4\|u_n\|_{H_0^1}\|\alpha_n g_n\|_{H_0^1}^3 + \frac{\beta}{4} C_1^4 \|\alpha_n g_n\|_{H_0^1}^4 \le \alpha_n^2 \widetilde{C}_R \|g_n\|_{H_0^1}^2, 
        \end{aligned}
    \end{equation}
    where $C_R$ and $\widetilde{C}_R$ are constants depending only on $\beta$, $\Omega$, $M$, $V_{\max}$, $C_u$, $C_g$, $C_1$, $C_2$, and $C_3$. Combining inequalities \eqref{equation_4_1} and \eqref{equation_4_2}, and choosing $C_{\alpha}\le 2^{-1}(C_R+\widetilde{C}_R)^{-1}$, we obtain
    \begin{align*}
        E(u_n) - E(u_{n+1}) &= E(u_n) - E(u_n-\alpha_n g_n) + E(\tilde{u}_n) - E(\tilde{u}_n+R_n)\\
        &\ge \Re(\nabla_{H_0^1} E(u_n), \alpha_n g_n)_{H_0^1} - \alpha_n^2(C_R+\widetilde{C}_R)\|g_n\|_{H_0^1}^2\\
        &= \alpha_n\left(1-\alpha_n(C_R+\widetilde{C}_R)\right)\|g_n\|_{H_0^1}^2 \ge \frac{\alpha_{\min}}{2} \|g_n\|_{H_0^1}^2. 
    \end{align*}
    Thus, (iii) holds for $n$. 

    Third, due to the norm equivalence in Lemma \ref{lemma_2_2} and the energy dissipation property (iii), we have
    \begin{align*}
        \|u_{n+1}\|_{H_0^1}^2 &\le \frac{1+K}{K} \|u_{n+1}\|_{a_0}^2 \le \frac{2(1+K)}{K} E(u_{n+1}) \le \frac{2(1+K)}{K} E(u_0)\\
        &\le \frac{1+K}{2K}\beta C_1^4 \|u_0\|_{H_0^1}^4 + \frac{1+K}{K}(1+V_{\max}C_3^2+\Omega M C_3)\|u_0\|_{H_0^1}^2. 
    \end{align*}
    By setting $C_u^2\ge \frac{1+K}{2K}\beta C_1^4 \|u_0\|_{H_0^1}^4 + \frac{1+K}{K}(1+V_{\max}C_3^2+\Omega M C_3)\|u_0\|_{H_0^1}^2$, (i) holds for $n+1$. In summary, by induction, the statements holds.
\end{proof}

To prove Theorem \ref{theorem_3_3}, we need to establish an auxiliary Theorem \ref{theorem_4_1}. Thanks to~\cite{chen2024convergence}*{Theorem 4.5}, we only need to add estimates for the rotation term.

\begin{theorem}\label{theorem_4_1}
    Let $\{u_n\}_{n=0}^\infty$ be an $H_0^1$-bounded sequence on $\mathcal{M}$ such that
    \begin{equation*}
        \lim_{n\to\infty} \left\|\nabla^{\mathcal{R}}_{H_0^1} E(u_n)\right\|_{H_0^1} = 0.
    \end{equation*}
    If $u^*$ is a weak limit point of $\{u_n\}_{n=0}^\infty$ in $H_0^1(\mathcal{D})$, then $u^*$ is a stationary point of the energy functional $E(u)$, and $\{u_n\}_{n=0}^\infty$ admits a subsequence converging strongly to $u^*$.
\end{theorem}

\begin{proof}
    Since the sequence $\{u_n\}_{n=0}^\infty$ is bounded in $H_0^1(\mathcal{D})$, by Rellich-Kondrachov compact embedding theorem, there exists a subsequence (still denoted by $\{u_n\}_{n=0}^\infty$ for simplicity) and a function $u^* \in H_0^1(\mathcal{D})$ such that as $n\to\infty$,
    \begin{equation*}
        u_n \to u^*,\quad\quad\text{weakly in $H_0^1(\mathcal{D})$, and strongly in $L^2(\mathcal{D})$ and $L^4(\mathcal{D})$.}
    \end{equation*}
    Furthermore, $u^* \in \mathcal{M}$. We analyze the convergence of each term in $\nabla_{H_0^1}^{\mathcal{R}} E(u_n)$ one by one.

    First, for the non-rotation terms, following~\cite{chen2024convergence}*{Theorem 4.5}, we have
    \begin{equation*}
        \mathcal{G}_{H_0^1}(u_n) \to \mathcal{G}_{H_0^1}(u^*), \quad\quad \mathcal{G}_{H_0^1}(Vu_n) \to \mathcal{G}_{H_0^1}(Vu^*), \quad\quad \mathcal{G}_{H_0^1}\left(\beta|u_n|^2 u_n\right) \to \mathcal{G}_{H_0^1}\left(\beta|u^*|^2 u^*\right)
    \end{equation*}
    strongly in $H_0^1(\mathcal{D})$ and $L^2(\mathcal{D})$. For the rotation term, by Lemma \ref{lemma_4_1}, we have
    \begin{align*}
        \|\mathcal{G}_{H_0^1}(L_z u_n-L_z u^*)\|_{H_0^1} &\le \|L_z(u_n-u^*)\|_{H^{-1}} = \sup_{\|w\|_{H_0^1}=1}\left|(L_z(u_n-u^*), w)_{L^2}\right| \\
        &= \sup_{\|w\|_{H_0^1}=1}\left|(u_n-u^*, L_z w)_{L^2}\right| \le \sup_{\|w\|_{H_0^1}=1}\int_{\mathcal{D}} |r||\nabla w||u_n-u^*|dr \\
        &\le \sup_{\|w\|_{H_0^1}=1}M\|w\|_{H_0^1}\|u_n-u^*\|_{L^2} = M\|u_n-u^*\|_{L^2} \to 0.
    \end{align*}
    Hence, $\mathcal{G}_{H_0^1}(\Omega L_z u_n) \to \mathcal{G}_{H_0^1}(\Omega L_z u^*)$ strongly in $H_0^1(\mathcal{D})$ and $L^2(\mathcal{D})$. For simplicity, let
    \begin{align*}
        \gamma_n &= \frac{1+\Re(u_n, \mathcal{G}_{H_0^1}(Vu_n+\beta|u_n|^2 u_n - \Omega L_z u_n))_{L^2}}{\|\mathcal{G}_{H_0^1} u_n\|_{H_0^1}^2}, \\
        \gamma^* &= \frac{1+\Re(u^*, \mathcal{G}_{H_0^1}(Vu^*+\beta|u^*|^2 u^* - \Omega L_z u^*))_{L^2}}{\|\mathcal{G}_{H_0^1} u^*\|_{H_0^1}^2}.
    \end{align*}
    Then, based on the convergence of each term, we have $\gamma_n \to \gamma^*$.

    In summary, $\nabla_{H_0^1}^{\mathcal{R}} E(u_n) \to \nabla_{H_0^1}^{\mathcal{R}} E(u^*)$ weakly in $H_0^1(\mathcal{D})$. Since $\lim\limits_{n\to\infty} \|\nabla^{\mathcal{R}}_{H_0^1} E(u_n)\|_{H_0^1} = 0$, we have
    \begin{equation*}
        \left\|\nabla^{\mathcal{R}}_{H_0^1} E(u^*)\right\|_{H_0^1}^2 = \lim_{n\to\infty}\left(\nabla^{\mathcal{R}}_{H_0^1} E(u^*), \nabla^{\mathcal{R}}_{H_0^1} E(u_n)\right)_{H_0^1} = 0.
    \end{equation*}
    This implies that $u^*$ is a stationary point of the functional $E(u)$.

    Finally, since $\lim\limits_{n\to\infty} \|\nabla^{\mathcal{R}}_{H_0^1} E(u_n)\|_{H_0^1} = 0$, and $\{u_n\}_{n=0}^{\infty}$ is bounded in $H_0^1(\mathcal{D})$, we have
    \begin{equation}\label{equation_4_4}
        \lim_{n\to\infty} \|u_n\|_{H_0^1}^2 + \left(Vu_n+\beta|u_n|^2 u_n - \Omega L_z u_n, u_n\right)_{L^2} - \gamma_n = \lim_{n\to\infty} \left(\nabla^{\mathcal{R}}_{H_0^1} E(u_n), u_n\right)_{H_0^1} = 0. 
    \end{equation}
    Since $u^*$ is a stationary point, we have 
    \begin{equation}\label{equation_4_6}
        \|u^*\|_{H_0^1}^2 + \left(Vu^*+\beta|u^*|^2 u^* - \Omega L_z u^*, u^*\right)_{L^2} - \gamma^* = \left(\nabla^{\mathcal{R}}_{H_0^1} E(u^*), u^*\right)_{H_0^1} = 0.
    \end{equation}
    Note that, based on the above convergences,
    \begin{equation}\label{equation_4_5}
        \lim_{n\to\infty} \left(Vu_n+\beta|u_n|^2 u_n - \Omega L_z u_n, u_n\right)_{L^2} - \gamma_n = \left(Vu^*+\beta|u^*|^2 u^* - \Omega L_z u^*, u^*\right)_{L^2} - \gamma^*.
    \end{equation}
    Combining equations \eqref{equation_4_4}, \eqref{equation_4_6} and \eqref{equation_4_5}, we obtain $\lim\limits_{n\to\infty} \|u_n\|_{H_0^1}^2 = \|u^*\|_{H_0^1}^2$. Together with the weak convergence of $u_n$ to $u^*$ in $H_0^1(\mathcal{D})$, we have
    \begin{equation*}
        \lim_{n\to\infty} \|u_n-u^*\|_{H_0^1}^2 = \lim_{n\to\infty} \|u_n\|_{H_0^1}^2 + \|u^*\|_{H_0^1}^2 - 2\Re(u_n, u^*)_{H_0^1} = \|u^*\|_{H_0^1}^2 + \|u^*\|_{H_0^1}^2 - 2\|u^*\|_{H_0^1}^2 = 0, 
    \end{equation*}
    which means that $u_n$ converges strongly to $u^*$ in $H_0^1(\mathcal{D})$.
\end{proof}

Finally, we can obtain Theorem \ref{theorem_3_3} for the rotating case.

\begin{proof}[Proof of Theorem \ref{theorem_3_3}]
    According to Theorem \ref{theorem_3_1}, we have
    \begin{equation*}
        \sum_{n=0}^{\infty} \left\|\nabla^{\mathcal{R}}_{H_0^1} E(u_n)\right\|_{H_0^1}^2 \le \frac{2}{\alpha_{\min}} \sum_{n=0}^{\infty}\left(E(u_n)-E(u_{n+1})\right) \le \frac{2E(u_0)}{\alpha_{\min}} < \infty,
    \end{equation*}
    which implies
    \begin{equation*}
        \lim_{n\to\infty} \left\|\nabla^{\mathcal{R}}_{H_0^1} E(u_n)\right\|_{H_0^1} = 0.
    \end{equation*}
    Similarly, based on Theorem \ref{theorem_3_1}, $\{u_n\}_{n=0}^{\infty}$ is bounded. The proof is then completed by applying Theorem \ref{theorem_4_1}.
\end{proof}

\begin{remark}
    (i) As noted at the beginning of this section, the proof presented here addresses only the global convergence of the $H_0^1$ scheme. The argument for the $a_0$ scheme is analogous. \\
    (ii) The results of this section extend the non-rotating case studied in~\cite{chen2024convergence}. However, local convergence does not follow directly. To handle this, we introduce the framework of quotient spaces and quotient metrics in the subsequent analysis.
\end{remark}

\section{Proof of Local Convergence}\label{section_5}

In this section, we will prove Theorem \ref{theorem_3_5} and Theorem \ref{theorem_3_7}. Theorem \ref{theorem_3_6} can be proved using the same method as Theorem \ref{theorem_3_5}. First, for $u\in\mathcal{M}$, the second-order Fréchet derivative of the Gross--Pitaevskii energy functional $E(u)$ at $u$ is given by
\begin{equation*}
    \langle E''(u)v, w \rangle := \Re(\nabla v, \nabla w)_{L^2} + \Re( Vv + \beta|u|^2v - \Omega L_z v + 2\beta \Re(\overline{u}v)u, w)_{L^2}\quad\quad\forall v,w\in H_0^1(\mathcal{D}).
\end{equation*}

\begin{lemma}\label{lemma_5_1}
    Let 
        \begin{align*}
            \mu =& \frac{K(\lambda_2-\lambda)}{4(1+K)\lambda_2}C_3\|u^*\|_{H_0^1} + \beta C_4^3\|u^*\|_{H_0^1} + \beta C_2C_3\|u^*\|_{L^4}^3\left(1+C_3\|u^*\|_{H_0^1}\right).
        \end{align*}
    and assume $u \in \mathcal{M}$ is close enough to $u^*$ such that $\rho_1(u) \le \mu^{-1} \frac{K(\lambda_2-\lambda)}{4(1+K)\lambda_2}$, then we have
        \begin{equation*}
            E(u) - E(u^*) \ge \left(\frac{K(\lambda_2-\lambda)}{4(1+K)\lambda_2} - \mu\rho_1(u)\right) \rho_2^2(u). 
        \end{equation*}
\end{lemma}

\begin{proof}
    Since the energy functional $E(u)$ and the metrics $\rho_1(u)$, $\rho_2(u)$ are phase-independent, i.e., for any $\omega$ and $u$,
    \begin{equation*}
        E(u) = E(\exp(\mathrm{i}\omega)u),\quad\quad \rho_1(u) = \rho_1(\exp(\mathrm{i}\omega)u),\quad\quad \rho_2(u) = \rho_2(\exp(\mathrm{i}\omega)u).
    \end{equation*}
    Thus, without loss of generality, we can assume that $\argmin_\omega \|u-e^{\mathrm{i}\omega}u^*\|_{L^2} = 0$, which implies
    \begin{equation}\label{equation_5_1}
        \Re(u-u^*, \mathrm{i}u^*)_{L^2} = -\frac{1}{2}\frac{d}{dw} \|u-e^{\mathrm{i}\omega}u^*\|_{L^2}^2 \Big|_{w=0} = 0.
    \end{equation}
    Let $\varphi=u-u^*$. Since $u,u^*\in \mathcal{M}$, by the normalization condition, $\|\varphi+u^*\|_{L^2}=\|u\|_{L^2}=\|u^*\|_{L^2}=1$, we have
    \begin{equation}\label{equation_5_2}
        \Re(\varphi, u^*)_{L^2} = \frac{1}{2}\left(\|\varphi+u^*\|_{L^2}^2 - \|u^*\|_{L^2}^2 - \|\varphi\|_{L^2}^2\right) = - \frac{1}{2} \|\varphi\|_{L^2}^2,
    \end{equation}
    Combining equations \eqref{equation_5_1} and \eqref{equation_5_2}, we obtain
    \begin{equation*}
        (\varphi, u^*)_{L^2} = \Re(\varphi, u^*)_{L^2} - \mathrm{i}\Re(\varphi, \mathrm{i}u^*)_{L^2} = -\frac{1}{2}\|\varphi\|_{L^2}^2.
    \end{equation*}
    Therefore, setting $\eta = -\frac{1}{2}\|\varphi\|_{L^2}^2$, we can decompose $\varphi$ as $\varphi = \eta u^* + \psi$, where $(\psi, u^*)_{L^2}=0$, and thus $\Re(\psi, iu^*)=\Re(\psi, u^*)=0$, implying $\psi \in T_{u^*}\mathcal{M}\cap T_{\mathrm{i}u^*}\mathcal{M}$. Let
    \begin{equation*}
        \mathcal{E}(u) := E(u) - \frac{\lambda}{2}\|u\|_{L^2}^2 = \int_\mathcal{D} \left(\frac{1}{2}|\nabla u|^2 + \frac{1}{2}V|u|^2 + \frac{\beta}{4}|u|^4 - \frac{\Omega}{2}\overline{u}L_z u\right) dr - \frac{\lambda}{2}\int_\mathcal{D} |u|^2 dr, 
    \end{equation*}
    According to the normalization condition, we have
    \begin{equation}\label{equation_5_9}
        \begin{aligned}
            E(u) - E(u^*) =& \mathcal{E}(u) - \mathcal{E}(u^*)\\
            =& \left[\langle E'(u^*), \varphi \rangle - \lambda\Re(u^*, \varphi)_{L^2}\right] + \frac{1}{2}\left[\langle E''(u^*)\varphi, \varphi \rangle - \lambda\Re(\varphi, \varphi)_{L^2}\right] \\
            &\quad + \beta\int_\mathcal{D} |\varphi|^2\Re(\overline{u^*}\varphi)dr + \frac{\beta}{4}\int_\mathcal{D}|\varphi|^4 dr.
        \end{aligned}
    \end{equation}
    The first-order derivative term $\langle E'(u^*), \varphi \rangle - \lambda\Re(u^*, \varphi)_{L^2} = 0$, and the final quartic term $\int_\mathcal{D} |\varphi|^4 dr \ge 0$. We proceed to estimate the remaining two terms.

    For the second term on the right-hand side of equation \eqref{equation_5_9}, using the decomposition of $\varphi$, we have
    \begin{equation}\label{equation_5_25}
        \begin{aligned}
            \langle E''(u^*)\varphi, \varphi \rangle - \lambda\Re(\varphi, \varphi)_{L^2} =& \langle E''(u^*)\psi, \psi \rangle - \lambda\Re(\psi, \psi)_{L^2} \\
            &\quad + 2\eta\left[\langle E''(u^*)u^*, \psi \rangle - \lambda\Re(u^*, \psi)_{L^2}\right] \\
            &\quad + \eta^2\left[\langle E''(u^*)u^*, u^* \rangle - \lambda\Re(u^*, u^*)_{L^2}\right].
        \end{aligned}
    \end{equation}
    For the first term on the right-hand side (RHS) of \eqref{equation_5_25}, since $\psi \in T_{u^*}\mathcal{M}\cap T_{iu^*}\mathcal{M}$, by Lemma \ref{lemma_3_1},
    \begin{equation}\label{equation_5_3}
        \langle E''(u^*)\psi, \psi \rangle - \lambda\Re(\psi, \psi)_{L^2} = \langle E''(u^*)\psi, \psi \rangle - \lambda\|\psi\|_{L^2}^2 \ge \frac{K(\lambda_2-\lambda_1)}{2(1+K)\lambda_2}\|\psi\|_{H_0^1}^2.
    \end{equation}
    For the second term on the RHS of \eqref{equation_5_25}, since $u^*$ is a solution to the eigenvalue equation \eqref{equation_2_7}, by Lemma \ref{lemma_2_1},
    \begin{equation}\label{equation_5_4}
        \begin{aligned}
            \langle E''(u^*)u^*, \psi \rangle - \lambda\Re(u^*, \psi)_{L^2} =& 2\beta \Re\int_\mathcal{D} |u^*|^2 \overline{u^*} \psi dr \le 2\beta |\langle |u^*|^2u^*, \psi \rangle| \\
            \le& 2\beta\||u^*|^2u^*\|_{H^{-1}}\|\psi\|_{H_0^1} \le 2\beta C_2 \|u^*\|_{L^4}^3 \|\psi\|_{H_0^1}.
        \end{aligned}
    \end{equation}
    For the third term on the RHS of \eqref{equation_5_25}, we have
    \begin{equation}\label{equation_5_5}
        \langle E''(u^*)u^*, u^* \rangle - \lambda\Re(u^*, u^*)_{L^2} = 2\beta\int_\mathcal{D} |u^*|^4 dr = 2\beta\|u^*\|_{L^4}^4 \ge 0.
    \end{equation}
    Substituting inequalities \eqref{equation_5_3}-\eqref{equation_5_5} into equation \eqref{equation_5_25}, and noting that $\eta<0$, we obtain
    \begin{equation}\label{equation_5_6}
        \langle E''(u^*)\varphi, \varphi \rangle - \lambda\Re(\varphi, \varphi)_{L^2} \ge \frac{K(\lambda_2-\lambda)}{2(1+K)\lambda_2}\|\psi\|_{H_0^1}^2 - 2\beta C_2\|u^*\|_{L^4}^3\|\psi\|_{H_0^1}\|\varphi\|_{L^2}^2.
    \end{equation}
    Furthermore, $\psi = \varphi + \frac{1}{2}\|\varphi\|_{L^2}^2 u^*$, so
    \begin{align}
        \label{equation_5_7} \|\psi\|_{H_0^1} &\ge \|\varphi\|_{H_0^1} - \frac{1}{2}\|\varphi\|_{L^2}^2\|u^*\|_{H_0^1} \ge \|\varphi\|_{H_0^1}\left(1-\frac{1}{2}C_3\|\varphi\|_{L^2}\|u^*\|_{H_0^1}\right),\\
        \label{equation_5_8} \|\psi\|_{H_0^1} &\le \|\varphi\|_{H_0^1} + \frac{1}{2}\|\varphi\|_{L^2}^2\|u^*\|_{H_0^1} \le \|\varphi\|_{H_0^1}\left(1+\frac{1}{2}C_3\|\varphi\|_{L^2}\|u^*\|_{H_0^1}\right).
    \end{align}
    Combining inequalities \eqref{equation_5_6}-\eqref{equation_5_8}, we finally obtain the lower bound estimate for equation \eqref{equation_5_25}:
    \begin{equation}\label{equation_5_10}
        \begin{aligned}
            \langle E''(u^*)\varphi, \varphi \rangle - \lambda\Re(\varphi, \varphi)_{L^2} \ge& \frac{K(\lambda_2-\lambda)}{2(1+K)\lambda_2}\|\varphi\|_{H_0^1}^2 \left(1-\frac{1}{2}C_3\|\varphi\|_{L^2}\|u^*\|_{H_0^1}\right)^2 \\
            &\quad - 2\beta C_2\|u^*\|_{L^4}^3\|\varphi\|_{H_0^1}\left(1+\frac{1}{2}C_3\|\varphi\|_{L^2}\|u^*\|_{H_0^1}\right)\|\varphi\|_{L^2}^2 \\
            \ge& \frac{K(\lambda_2-\lambda)}{2(1+K)\lambda_2}\left(1-\frac{1}{2}C_3\|\varphi\|_{L^2}\|u^*\|_{H_0^1}\right)^2 \|\varphi\|_{H_0^1}^2 \\
            &\quad - 2\beta C_2C_3\|u^*\|_{L^4}^3\left(1+C_3\|u^*\|_{H_0^1}\right)\|\varphi\|_{L^2} \|\varphi\|_{H_0^1}^2 \\
            \ge& \left(\frac{K(\lambda_2-\lambda)}{2(1+K)\lambda_2} - 2\mu\rho_1(u) + 2\beta C_4^3\|u^*\|_{H_0^1}\rho_1(u)\right)\|\varphi\|_{H_0^1}^2.
        \end{aligned}
    \end{equation}
    The second inequality uses Poincaré's inequality and the triangle inequality, and the third inequality omits the quadratic term of $\|\varphi\|_{L^2}$ from the squared term.

    For the cubic term of $\varphi$ in equation \eqref{equation_5_9}, using Hölder's inequality, we obtain
    \begin{equation}\label{equation_5_11}
        \begin{aligned}
            \beta \int_\mathcal{D} |\varphi|^2\Re(\overline{u^*}\varphi)dr &\ge -\beta\int_\mathcal{D} |\varphi|^3|u^*|dr \ge -\beta \|u^*\|_{L^6}\|\varphi\|_{L^6}^2\|\varphi\|_{L^2} \\
            &\ge -\beta C_4^3\|u^*\|_{H_0^1}\|\varphi\|_{H_0^1}^2\rho_1(u).
        \end{aligned}
    \end{equation}

    Combining equations \eqref{equation_5_9}, \eqref{equation_5_10} and \eqref{equation_5_11}, we finally get
    \begin{align*}
        E(u)-E(u^*) &\ge \left(\frac{K(\lambda_2-\lambda)}{4(1+K)\lambda_2} - \mu\rho_1(u)\right) \|\varphi\|_{H_0^1}^2 = \left(\frac{K(\lambda_2-\lambda)}{4(1+K)\lambda_2} - \mu\rho_1(u)\right) \|u-u^*\|_{H_0^1}^2 \\
        &\ge \left(\frac{K(\lambda_2-\lambda)}{4(1+K)\lambda_2} - \mu\rho_1(u)\right) \min_{\omega\in [-\pi, \pi)}\|u-e^{\mathrm{i}\omega}u^*\|_{H_0^1}^2 = \left(\frac{K(\lambda_2-\lambda)}{4(1+K)\lambda_2} - \mu\rho_1(u)\right) \rho_2^2(u). 
    \end{align*}
    This completes the proof.
\end{proof}

\begin{lemma}\label{lemma_5_2}
    Let $C_u$ be the constant defined in Theorem \ref{theorem_3_1}, and $C_{\rho} < \frac{1}{2}C_3^{-3}\|\mathcal{G}_{H_0^1} u^*\|_{H_0^1}^2$ be a positive constant. Then there exists a constant $L$ depending only on $D$, $d$, $\beta$, $V$, $\Omega$, $\|u^*\|_{H_0^1}$, and $C_{\rho}$ such that when $u \in \mathcal{M}$ satisfies $\|u\|_{H_0^1} \le C_u$ and $\rho_2(u)\le C_{\rho}$, we have
    \begin{equation*}
        \left\|\nabla_{H_0^1}^{\mathcal{R}} E(u)\right\|_{H_0^1} \le L \rho_2(u). 
    \end{equation*}
\end{lemma}

\begin{proof}
    For any $\omega$ and $u$, we have $\left\|\nabla_{H_0^1}^{\mathcal{R}} E(\exp(\mathrm{i}\omega)u)\right\|_{H_0^1}=\left\|\exp(\mathrm{i}\omega)\nabla_{H_0^1}^{\mathcal{R}} E(u)\right\|_{H_0^1}=\left\|\nabla_{H_0^1}^{\mathcal{R}} E(u)\right\|_{H_0^1}$. Therefore, without loss of generality, we can assume that $\argmin\limits_{\omega \in [-\pi,\pi)} \|u-e^{\mathrm{i}\omega}u^*\|_{H_0^1}=0$, which implies $\rho_2(u) = \|u-u^*\|_{H_0^1}$. Let
    \begin{equation*}
        \gamma(u) = \frac{1+\Re(u, \mathcal{G}_{H_0^1}(Vu+\beta|u|^2 u - \Omega L_z u))_{L^2}}{\|\mathcal{G}_{H_0^1} u\|_{H_0^1}^2}, \quad\quad \gamma^* = \gamma(u^*).
    \end{equation*}
    We have
    \begin{equation}\label{equation_5_12}
        \begin{aligned}
            &\left\|\nabla_{H_0^1}^{\mathcal{R}} E(u)\right\|_{H_0^1} = \left\|\nabla_{H_0^1}^{\mathcal{R}} E(u) - \nabla_{H_0^1}^{\mathcal{R}} E(u^*)\right\|_{H_0^1} \\
            =& \Big\| \left(u + \mathcal{G}_{H_0^1}(Vu+\beta|u|^2 u - \Omega L_z u) - \gamma(u) \mathcal{G}_{H_0^1} u\right) \\
            &\quad - \left(u^* + \mathcal{G}_{H_0^1}(Vu^*+\beta|u^*|^2 u^* - \Omega L_z u^*) - \gamma^* \mathcal{G}_{H_0^1} u^*\right) \Big\|_{H_0^1} \\
            \le& \rho_2(u) + \left\|\mathcal{G}_{H_0^1}(V(u-u^*))\right\|_{H_0^1} + \beta\left\|\mathcal{G}_{H_0^1}\left(|u|^2 u-|u^*|^2 u^*\right)\right\|_{H_0^1} \\
            &\quad + \Omega\left\|\mathcal{G}_{H_0^1}(L_z(u-u^*))\right\|_{H_0^1} + \left\|\gamma(u)\mathcal{G}_{H_0^1}u - \gamma^*\mathcal{G}_{H_0^1}u^*\right\|_{H_0^1}.
        \end{aligned}
    \end{equation}
    From the estimates in Lemma \ref{lemma_4_1} and Theorem \ref{theorem_4_1}, we have
    \begin{align}
        \label{equation_5_13} \left\|\mathcal{G}_{H_0^1}(V(u-u^*))\right\|_{H_0^1} &\le C_3\|V(u-u^*)\|_{L^2} \le C_3 V_{\max}\|u-u^*\|_{L^2} \le C_3^2 V_{\max}\rho_2(u), \\
        \label{equation_5_14} \left\|\mathcal{G}_{H_0^1}(L_z(u-u^*))\right\|_{H_0^1} &\le M \|u-u^*\|_{L^2} \le M C_3 \rho_2(u).
    \end{align}
    For the nonlinear term, by Lemma \ref{lemma_2_1}, Lemma \ref{lemma_4_1} and Hölder's inequality, we obtain
    \begin{equation}\label{equation_5_15}
        \begin{aligned}
            \left\|\mathcal{G}_{H_0^1}\left(|u|^2 u-|u^*|^2 u^*\right)\right\|_{H_0^1} &\le \left\||u|^2 u - |u^*|^2 u^*\right\|_{H^{-1}} \le C_2 \left\||u|^2 u - |u^*|^2 u^*\right\|_{L^{4/3}} \\
            &= C_2 \left\||u|^2(u-u^*) + \left(|u|^2-|u^*|^2\right)u^*\right\|_{L^{4/3}} \\
            &\le C_2 \left( \left\||u|^2(u-u^*)\right\|_{L^{4/3}} + \left\|\left(|u|-|u^*|\right)\left(|u|+|u^*|\right)u^*\right\|_{L^{4/3}} \right) \\
            &\le C_2 \left( \|u\|_{L^4}^2\|u-u^*\|_{L^4} + \|u-u^*\|_{L^4}(\|u\|_{L^4}+\|u^*\|_{L^4})\|u^*\|_{L^4} \right) \\
            &\le C_1^3 C_2 \left(\|u\|_{H_0^1}^2 + \|u\|_{H_0^1}\|u^*\|_{H_0^1} + \|u^*\|_{H_0^1}^2 \right) \rho_2(u) \\
            &\le L_1 \rho_2(u), 
        \end{aligned}
    \end{equation}
    where
    \begin{equation*}
        L_1 = C_1^3 C_2 \left(C_u^2 + C_u\|u^*\|_{H_0^1} + \|u^*\|_{H_0^1}^2 \right).
    \end{equation*}
    For the last term, note that
    \begin{equation*}
        0 = \nabla_{H_0^1}^{\mathcal{R}}E(u^*) = \mathcal{G}_{H_0^1}\left(-\Delta u^*+Vu^*+\beta|u^*|^2u^*-\Omega L_z u^*\right) - \gamma^* \mathcal{G}_{H_0^1}u^* = \left(\lambda-\gamma^*\right)\mathcal{G}_{H_0^1}u^*,
    \end{equation*}
    which implies $\gamma^*=\lambda$. Thus, by Lemma \ref{lemma_4_1}, we have
    \begin{equation}\label{equation_5_16}
        \begin{aligned}
            \left\|\gamma(u)\mathcal{G}_{H_0^1}u - \gamma^*\mathcal{G}_{H_0^1}u^*\right\|_{H_0^1} &\le \left\|\left(\gamma(u)-\gamma^*\right)\mathcal{G}_{H_0^1}u\right\|_{H_0^1} + \left\|\gamma^*\left(\mathcal{G}_{H_0^1}u - \mathcal{G}_{H_0^1}u^*\right)\right\|_{H_0^1} \\
            &\le |\gamma(u)-\gamma^*| \left\|\mathcal{G}_{H_0^1}u\right\|_{H_0^1} + \lambda \left\|\mathcal{G}_{H_0^1}(u-u^*)\right\|_{H_0^1} \\
            &\le C_3|\gamma(u)-\gamma^*| + \lambda C_3^2\rho_2(u).
        \end{aligned}
    \end{equation}

    Now we estimate $|\gamma(u)-\gamma^*|$. Denote $A(u) = \Re(u, \mathcal{G}_{H_0^1}(Vu+\beta|u|^2 u - \Omega L_z u))_{L^2}$ and $B(u) = \|\mathcal{G}_{H_0^1} u\|_{H_0^1}^2$. Then $\gamma(u) = (1+A(u))/B(u)$.

    First, consider $A(u)$. 
    \begin{equation}\label{equation_5_17}
        \begin{aligned}
            |A(u)-A(u^*)| \le& \left|\Re(\mathcal{G}_{H_0^1}(Vu), u-u^*)_{H_0^1}\right| + \left|\Re(\mathcal{G}_{H_0^1}(V(u-u^*)), u^*)_{H_0^1}\right| \\
            &\quad + \beta\left|\Re(\mathcal{G}_{H_0^1}(|u|^2 u), u-u^*)_{H_0^1}\right| + \beta\left|\Re(\mathcal{G}_{H_0^1}(|u|^2 u-|u^*|^2 u^*), u^*)_{H_0^1}\right| \\
            &\quad + \Omega\left|\Re(\mathcal{G}_{H_0^1}(L_z u), u-u^*)_{H_0^1}\right| + \Omega\left|\Re(\mathcal{G}_{H_0^1}(L_z(u-u^*)), u^*)_{H_0^1}\right| \\
            \le& \left\|\mathcal{G}_{H_0^1}(Vu)\right\|_{H_0^1}\|u-u^*\|_{H_0^1} + \left\|\mathcal{G}_{H_0^1}(V(u-u^*))\right\|_{H_0^1}\|u^*\|_{H_0^1} \\
            &\quad + \beta\left\|\mathcal{G}_{H_0^1}(|u|^2 u)\right\|_{H_0^1}\|u-u^*\|_{H_0^1} + \beta\left\|\mathcal{G}_{H_0^1}(|u|^2 u-|u^*|^2 u^*)\right\|_{H_0^1}\|u^*\|_{H_0^1} \\
            &\quad + \Omega\left\|\mathcal{G}_{H_0^1}(L_z u)\right\|_{H_0^1}\|u-u^*\|_{H_0^1} + \Omega\left\|\Re(\mathcal{G}_{H_0^1}(L_z(u-u^*))\right\|_{H_0^1}\|u^*\|_{H_0^1} \\
            \le& L_2 \rho_2(u), 
        \end{aligned}
    \end{equation}
    where
    \begin{equation*}
        L_2 = C_3 V_{\max} + C_3^2 V_{\max}\|u^*\|_{H_0^1} + \beta C_1^3 C_2 C_u^3 + \beta L_1\|u^*\|_{H_0^1} + \Omega M + \Omega M C_3\|u^*\|_{H_0^1}. 
    \end{equation*}
    Next, consider $B(u)$. 
    \begin{equation}\label{equation_5_18}
        \begin{aligned}
            |B(u)-B(u^*)| &= \left|\|\mathcal{G}_{H_0^1}u\|_{H_0^1}^2 - \|\mathcal{G}_{H_0^1}u^*\|_{H_0^1}^2\right| \\
            &= \left(\|\mathcal{G}_{H_0^1}u\|_{H_0^1} + \|\mathcal{G}_{H_0^1}u^*\|_{H_0^1}\right) \Big|\|\mathcal{G}_{H_0^1}u\|_{H_0^1} - \|\mathcal{G}_{H_0^1}u^*\|_{H_0^1}\Big|\\
            &\le \left(\|\mathcal{G}_{H_0^1}u\|_{H_0^1} + \|\mathcal{G}_{H_0^1}u^*\|_{H_0^1}\right) \|\mathcal{G}_{H_0^1}(u-u^*)\|_{H_0^1} \\
            &\le 2C_3^3\rho_2(u).
        \end{aligned}
    \end{equation}
    Therefore, from inequalities \eqref{equation_5_17} and \eqref{equation_5_18}:
    \begin{equation}\label{equation_5_19}
        \begin{aligned}
            |\gamma(u)-\gamma^*| &\le \frac{(1+A(u^*))|B(u^*)-B(u)| + |A(u^*)-A(u)|B(u^*)}{\left(B(u^*)-|B(u)-B(u^*)|\right)B(u^*)} \\
            &\le \frac{2C_3^3(1+A(u^*)) + L_2B(u^*)}{\left(B(u^*)-2C_3^3\rho_2(u)\right)B(u^*)}\rho_2(u) \\
            &\le \frac{2C_3^3\lambda + L_2}{B(u^*)-2C_3^3C_{\rho}}\rho_2(u).
        \end{aligned}
    \end{equation}
    Finally, combining inequalities \eqref{equation_5_12}-\eqref{equation_5_16} and \eqref{equation_5_19}, we obtain
    \begin{equation*}
        \left\|\nabla_{H_0^1}^{\mathcal{R}} E(u)\right\|_{H_0^1} \le L \rho_2(u),
    \end{equation*}
    where
    \begin{equation*}
        L = 1 + C_3^2 V_{\max} + \beta L_1 + \Omega M C_3 + C_3\frac{2C_3^3\lambda + L_2}{B(u^*)-2C_3^3C_{\rho}} + \lambda C_3^2.
    \end{equation*}
    The proof is complete. 
\end{proof}

\begin{remark}\label{remark_5_1}
    Similar to~\cite{chen2024convergence}*{Lemma B.2}, we can prove that when $u\in\mathcal{M}$ satisfies $\|u\|_{H_0^1} \le C_{u}$ and $\rho_2(u)$ is sufficiently small, there exists a constant $L$ depending only on $D$, $d$, $\beta$, $V$, $\Omega$, and $\|u^*\|_{H_0^1}$ such that 
    \begin{equation*}
        \|\nabla_{a_u}^{\mathcal{R}} E(u)\|_{a_{u^*}} \le L\|u-u^*\|_{a_{u^*}}. 
    \end{equation*}
    This result will be used in the proof of Theorem \ref{theorem_3_7}. 
\end{remark}

Now, we prove Theorem \ref{theorem_3_5}.

\begin{proof}[Proof of Theorem \ref{theorem_3_5}]
    We proceed by induction. Define  
    \begin{align*}
        \omega_n = \argmin_{\omega \in [-\pi,\pi)} \|u_n - e^{\mathrm{i}\omega}u^*\|_{H_0^1}, 
        \qquad u_n^* = e^{\mathrm{i}\omega_n}u^*, 
        \qquad e_n = u_n^* - u_n, 
        \qquad \delta_n = \rho_2(u_n)
    \end{align*}
    and hence $\delta_n = \|e_n\|_{H_0^1}$. We claim that $\delta_n \leq \theta^n \delta_0$ for some constant $0 < \theta < 1$, whose precise value will be determined below. This holds trivially for $n=0$. Now assume it holds for some fixed $n$, and we show it also holds for $n+1$.

    Let $L$ and $C_\rho$ be the constants from Lemma~\ref{lemma_5_2}. As long as $\delta_0 \leq C_\rho$, then by the inductive hypothesis we also have $\delta_n \leq \delta_0 \leq C_\rho$.  Hence, Lemma \ref{lemma_5_2} is applicable and we have 
    \begin{equation}\label{equation_5_35}
        \begin{aligned}
            &\left\|(u_n - u_n^*) - \alpha_n \nabla_{H_0^1}^{\mathcal{R}}E(u_n)\right\|_{H_0^1}^2 \\
            =& \|u_n - u_n^*\|_{H_0^1}^2 - 2\alpha_n \Re\left(u_n-u_n^*, \nabla_{H_0^1}^{\mathcal{R}}E(u_n)\right)_{H_0^1} + \alpha_n^2\left\|\nabla_{H_0^1}^{\mathcal{R}}E(u_n)\right\|_{H_0^1}^2 \\
            \le& (1+L^2\alpha_n^2)\delta_n^2 + 2\alpha_n\Re\left(e_n, \nabla_{H_0^1}E(u_n)\right)_{H_0^1} + 2\alpha_n\Re\left(e_n, \nabla_{H_0^1}^{\mathcal{R}}E(u_n)-\nabla_{H_0^1}E(u_n)\right)_{H_0^1}.
        \end{aligned}
    \end{equation}
    Note that
    \begin{equation}\label{equation_5_20}
        \begin{aligned}
            E(u^*)  -E(u_n) = E(u_n^*) -E(u_n) =& \langle E'(u_n), e_n \rangle + \frac{1}{2}\langle E''(u_n^*)e_n, e_n \rangle + \beta\int_\mathcal{D}|e_n|^2\Re(\overline{u_n}e_n) dr + \frac{\beta}{4}\int_\mathcal{D}|e_n|^4 dr \\
            &\quad + \frac{\beta}{2}\int_\mathcal{D}\left(|u_n|^2-|u_n^*|^2\right)|e_n|^2 dr + \beta\int_\mathcal{D}\left(\Re^2(\overline{u_n}e_n)-\Re^2(\overline{u_n^*}e_n)\right)dr.
        \end{aligned}
    \end{equation}
    Similar to Lemma \ref{lemma_5_1}, we now discuss each part of the right-hand side of \eqref{equation_5_20}. For the first term on the RHS of \eqref{equation_5_20}, by definition, we have
    \begin{equation}\label{equation_5_29}
        \langle E'(u_n), e_n \rangle = \Re\left(e_n, \nabla_{H_0^1}E(u_n)\right)_{H_0^1}.
    \end{equation}

    We now estimate the second term on the RHS of \eqref{equation_5_20}. From the normalization condition, $\|u_n^*-e_n\|_{L^2}=\|u_n\|_{L^2}=\|u_n^*\|_{L^2}=1$, so 
    \begin{equation*}
        \Re(u_n^*, e_n)_{L^2} = \frac{1}{2}\left(\|u_n^*\|_{L^2}^2 + \|e_n\|_{L^2}^2 - \|u_n^*-e_n\|_{L^2}^2\right) = \frac{1}{2}\|e_n\|_{L^2}^2.
    \end{equation*}
    Therefore, let $\eta = \frac{1}{2}\|e_n\|_{L^2}^2$. Then $e_n$ can be decomposed as $e_n = \eta u_n^*+\psi$, where $\Re(u_n^*, \psi)_{L^2}=0$, i.e., $\psi \in T_{u_n^*}\mathcal{M}$. Using this decomposition, the second term on the RHS of \eqref{equation_5_20} can be written as
    \begin{equation}\label{equation_5_21}
        \langle E''(u_n^*)e_n, e_n \rangle = \langle E''(u_n^*)\psi, \psi \rangle + 2\eta\langle E''(u_n^*)u_n^*, \psi \rangle + \eta^2\langle E''(u_n^*)u_n^*, u_n^* \rangle.
    \end{equation}
    For the first term on the RHS of \eqref{equation_5_21}, by Lemma \ref{lemma_3_1}, we have
    \begin{equation}\label{equation_5_22}
        \langle E''(u_n^*)\psi, \psi \rangle \ge \lambda\|\psi\|_{L^2}^2.
    \end{equation}
    For the second term on the RHS of \eqref{equation_5_21}, since the ground state $u_n^*$ satisfies the Gross-Pitaevskii eigenvalue equation \eqref{equation_2_7}, we can obtain
    \begin{equation}\label{equation_5_23}
        \begin{aligned}
            \langle E''(u_n^*)u_n^*, \psi \rangle &= \lambda\Re(u_n^*, \psi)_{L^2} + 2\beta\Re\left(|u_n^*|^2 u_n^*, \psi\right)_{L^2} = 2\beta\Re\left(|u_n^*|^2 u_n^*, \psi\right)_{L^2} \\
            &\ge -2\beta\||u_n^*|^2 u_n^*\|_{H^{-1}}\|\psi\|_{H_0^1} \ge -2\beta C_2\|u_n^*\|_{L^4}^3\|\psi\|_{H_0^1}.
        \end{aligned}
    \end{equation}
    For the third term on the RHS of \eqref{equation_5_21}, similarly, we have
    \begin{equation}\label{equation_5_24}
        \langle E''(u_n^*)u_n^*, u_n^* \rangle = \lambda\|u_n^*\|_{L^2}^2 + 2\beta\|u_n^*\|_{L^4}^4 > 0.
    \end{equation}
    Substituting inequalities \eqref{equation_5_22}-\eqref{equation_5_24} into equation \eqref{equation_5_21}, and noting that $\eta = \frac{1}{2}\|e_n\|_{L^2}^2$, we obtain
    \begin{equation}\label{equation_5_26}
        \begin{aligned}
            \langle E''(u_n^*)e_n, e_n \rangle &\ge \lambda\|\psi\|_{L^2}^2 - 4\beta\eta C_2\|u_n^*\|_{L^4}^3\|\psi\|_{H_0^1} \\
            &= \lambda\|\psi\|_{L^2}^2 - 2\beta C_2\|u_n^*\|_{L^4}^3\|e_n\|_{L^2}^2\|\psi\|_{H_0^1}.
        \end{aligned}
    \end{equation}
    From the definition of $\psi=e_n-\eta u_n^*$, we can derive
    \begin{align}
        \label{equation_5_27} \|\psi\|_{L^2} &\ge \|e_n\|_{L^2}-\eta\|u_n^*\|_{L^2} \ge \|e_n\|_{L^2}\left(1-\frac{1}{2}C_3\|e_n\|_{H_0^1}\right), \\
        \label{equation_5_28} \|\psi\|_{H_0^1} &\le \|e_n\|_{H_0^1}+\eta\|u_n^*\|_{H_0^1} \le \|e_n\|_{H_0^1}\left(1+\frac{1}{2}C_3^2\|e_n\|_{H_0^1}\|u_n^*\|_{H_0^1}\right).
    \end{align}
    Combining inequalities \eqref{equation_5_26}-\eqref{equation_5_28}, we finally obtain the lower bound estimate for the second term on the right-hand side of \eqref{equation_5_20}:
    \begin{equation}\label{equation_5_30}
        \begin{aligned}
            \langle E''(u_n^*)e_n, e_n \rangle \ge& \lambda\|e_n\|_{L^2}^2\left(1-\frac{1}{2}C_3\|e_n\|_{H_0^1}\right)^2 \\
            &\quad - 2\beta C_2 C_3^2\|u_n^*\|_{L^4}^3\left(1+\frac{1}{2}C_3^2\|e_n\|_{H_0^1}\|u_n^*\|_{H_0^1}\right)\|e_n\|_{H_0^1}^3.
        \end{aligned}
    \end{equation}

    For the remaining terms of \eqref{equation_5_20}, we have
    \begin{align}
        \label{equation_5_31} &\int_\mathcal{D}|e_n|^2\Re(\overline{u_n}e_n)dr \ge -\int_\mathcal{D}|e_n|^3|u_n| dr \ge -\|e_n\|_{L^6}^3\|u_n\|_{L^2} \ge -C_4^3\|e_n\|_{H_0^1}^3, \\ 
        \label{equation_5_32} &\int_\mathcal{D}|e_n|^4 dr \ge 0, \\
        \label{equation_5_33} &\begin{aligned}
            \int_\mathcal{D}(|u_n|^2-|u_n^*|^2)|e_n|^2 dr &= \int_\mathcal{D}(-2\Re\left(\overline{u_n^*}e_n)+|e_n|^2\right)|e_n|^2 dr \\
            &\ge -2\int_\mathcal{D}|e_n|^2|\Re(\overline{u_n^*}e_n)| dr \ge -2\int_\mathcal{D}|e_n|^3|u_n^*| dr \\
            &\ge -2\|e_n\|_{L^6}^3\|u_n^*\|_{L^2} \ge -2C_4^3\|e_n\|_{H_0^1}^3, 
        \end{aligned} \\
        \label{equation_5_34} &\int_\mathcal{D}\left(\Re^2(\overline{u_n}e_n) - \Re^2(\overline{u_n^*}e_n)\right) dr = \int_\mathcal{D}\left(-2\Re(\overline{u_n^*}e_n)+|e_n|^2\right)|e_n|^2 dr \ge -2C_4^3\|e_n\|_{H_0^1}^3.
    \end{align} 
    Substituting inequalities \eqref{equation_5_29}, \eqref{equation_5_30}-\eqref{equation_5_34} back into \eqref{equation_5_20}, we finally obtain
    \begin{equation}\label{equation_5_38}
        \begin{aligned}
            E(u_n^*)-E(u_n) \ge& \Re\left(e_n, \nabla_{H_0^1}E(u_n)\right)_{H_0^1} + \frac{\lambda}{2}\left(1-\frac{1}{2}C_3\delta_n\right)^2\|e_n\|_{L^2}^2 \\
            &\quad - \beta C_2 C_3^2\|u_n^*\|_{L^4}^3\left(1+\frac{1}{2}C_3^2\|u_n^*\|_{H_0^1}\delta_n\right)\delta_n^3 - 4\beta C_4^3 \delta_n^3.
        \end{aligned}
    \end{equation}
    Next, we want to apply Lemma \ref{lemma_5_1}. Note here that the required condition $\rho_1(u_n) \le \mu^{-1} \frac{K(\lambda_2-\lambda)}{4(1+K)\lambda_2}$ is fulfilled if $\rho_2(u_0)=\delta_0$ is sufficiently small. This is seen by using the definition of $\rho_1(u_n)$ and the induction hypothesis, which give us
    \begin{equation}\label{equation_5_39}
        \rho_1(u_n) \,=\, \min_{\omega}\|u_n-e^{\mathrm{i}\omega}u^*\|_{L^2} \,\le\, \|e_n\|_{L^2} \le C_3 \|e_n\|_{H^1_0}  \,=\, C_3 \delta_n \,\le\, C_3 \delta_0.
    \end{equation} 
    Hence, for $\delta_0 \le C_3^{-1}\mu^{-1} \frac{K(\lambda_2-\lambda)}{4(1+K)\lambda_2}$, Lemma \ref{lemma_5_1} is applicable and yields
    \begin{equation*}
        E(u_n^*)-E(u_n) \le \left(\mu\rho_1(u_n)-\frac{K(\lambda_2-\lambda)}{4(1+K)\lambda_2}\right)\delta_n^2.
    \end{equation*}
    Therefore
    \begin{equation*}
        E(u_n^*)-E(u_n) \le \left(\mu C_3 \delta_0-\frac{K(\lambda_2-\lambda)}{4(1+K)\lambda_2}\right)\delta_n^2. 
    \end{equation*}
    Combining \eqref{equation_5_38} and \eqref{equation_5_39}, we obtain the estimate
    \begin{equation}\label{equation_5_36}
        \begin{aligned}
            \Re\left(e_n, \nabla_{H_0^1}E(u_n)\right)_{H_0^1} \le& -\frac{\lambda}{2}\left(1-\frac{1}{2}C_3\delta_n\right)^2\|e_n\|_{L^2}^2 + \left(\mu C_3 \delta_0-\frac{K(\lambda_2-\lambda)}{4(1+K)\lambda_2}\right)\delta_n^2 \\
            &\quad + \beta C_2 C_3^2\|u_n^*\|_{L^4}^3\left(1+\frac{1}{2}C_3^2\|u_n^*\|_{H_0^1}\delta_n\right)\delta_n^3 + 4\beta C_4^3 \delta_n^3.
        \end{aligned}
    \end{equation}

    On the other hand, similar to Lemma \ref{lemma_5_2}, we denote 
    \begin{equation*}
        \gamma(u) = \frac{1+\Re\left(u, \mathcal{G}_{H_0^1}(Vu+\beta|u|^2u-\Omega L_z u)\right)_{L^2}}{\|\mathcal{G}_{H_0^1}u\|_{H_0^1}^2}, \quad \gamma_n=\gamma(u_n), \quad \gamma_n^*=\gamma(u_n^*)=\gamma(\exp(\mathrm{i}\omega_n)u^*)=\lambda. 
    \end{equation*}
    According to the definition of the projected gradient, we have
    \begin{equation}\label{equation_5_37}
        \begin{aligned}
            \Re\left(e_n, \nabla_{H_0^1}^{\mathcal{R}}E(u_n)-\nabla_{H_0^1}E(u_n)\right)_{H_0^1} =& -\gamma_n \Re\left(e_n, \mathcal{G}_{H_0^1}u_n\right)_{H_0^1} = -\gamma_n \Re(e_n, u_n)_{L^2} \\
            =& \frac{\gamma_n}{2}\left(\|u_n\|_{L^2}^2+\|e_n\|_{L^2}^2-\|u_n+e_n\|_{L^2}^2\right) = \frac{\gamma_n}{2}\|e_n\|_{L^2}^2.
        \end{aligned}
    \end{equation}
    Substituting \eqref{equation_5_36} and \eqref{equation_5_37} into \eqref{equation_5_35}, we obtain
    \begin{align*}
        &\|(u_n-u_n^*) - \alpha_n \nabla_{H_0^1}^{\mathcal{R}}E(u_n)\|_{H_0^1}^2 \\
        \le& \left(1+L^2\alpha_n^2\right)\delta_n^2 - \lambda\alpha_n\left(1-\frac{1}{2}C_3\delta_n\right)^2\|e_n\|_{L^2}^2 + \left(2\mu C_3 \delta_0-\frac{K(\lambda_2-\lambda)}{2(1+K)\lambda_2}\right)\alpha_n\delta_n^2 \\
        &\quad + 2\alpha_n\beta C_2 C_3^2\|u_n^*\|_{L^4}^3\left(1+\frac{1}{2}C_3^2\|u_n^*\|_{H_0^1}\delta_n\right)\delta_n^3 + 8 \alpha_n \beta C_4^3 \delta_n^3 + \alpha_n \gamma_n \|e_n\|_{L^2}^2.
    \end{align*}
    Choosing a positive constant $C_5 < \frac{K(\lambda_2-\lambda)}{2(1+K)\lambda_2} - 2\mu C_3 \delta_0$, and using Poincaré's inequality and the estimate for $\gamma_n$ in Lemma \ref{lemma_5_2}, we get 
    \begin{align*}
        &\|(u_n-u_n^*) - \alpha_n \nabla_{H_0^1}^{\mathcal{R}}E(u_n)\|_{H_0^1}^2 \\
        \le& \left(1+L^2\alpha_n^2 - C_5\alpha_n\right)\delta_n^2 + 2\alpha_n\beta C_2 C_3^2\|u_n^*\|_{L^4}^3\left(1+\frac{1}{2}C_3^2\|u_n^*\|_{H_0^1}\delta_n\right)\delta_n^3 + 8\alpha_n\beta C_4^3 \delta_n^3 \\
        &\quad + \left[ \gamma_n - \gamma_n^*\left(1-\frac{1}{2}C_3\delta_n\right)^2 + C_3^{-2}\left(2\mu C_3 \delta_0 - \frac{K(\lambda_2-\lambda)}{2(1+K)\lambda_2} + C_5\right)\right]\alpha_n \|e_n\|_{L^2}^2 \\
        \le& P_1(\alpha_n, \delta_n)\delta_n^2 + P_2(\delta_n)\alpha_n \|e_n\|_{L^2}^2. 
    \end{align*}
    where the polynomials $P_1$ and $P_2$ are
    \begin{align*}
        &P_1(\alpha, \delta) = L^2 \alpha^2 + \left[\beta C_2 C_3^4\|u_n^*\|_{L^4}^3\|u_n^*\|_{H_0^1}\delta^2 + 2\beta\left(C_2 C_3^2\|u_n^*\|_{L^4}^3 + 4 C_4^3\right)\delta - C_5\right]\alpha + 1, \\
        &P_2(\delta) = -\frac{1}{4}\lambda C_3^2 \delta^2 + \left(\frac{2C_3^3\lambda + L_2}{\|\mathcal{G}_{H_0^1} u^*\|_{H_0^1}^2-2C_3^3C_{\rho}} + \lambda C_3\right)\delta + C_3^{-2}\left(2\mu C_3 \delta_0 - \frac{K(\lambda_2-\lambda)}{2(1+K)\lambda_2} + C_5\right). 
    \end{align*}
    If $\delta_n$ satisfies $P_2(\delta_n) < 0$, then
    \begin{align*}
        \left\|(u_n-u_n^*) - \alpha_n \nabla_{H_0^1}^{\mathcal{R}}E(u_n)\right\|_{H_0^1}^2 \le P_1(\alpha_n, \delta_n)\delta_n^2 + P_2(\delta_n)\alpha_n \|e_n\|_{L^2}^2 < P_1(\alpha_n, \delta_n)\delta_n^2. 
    \end{align*}

    On the other hand, let $g_n = \nabla_{H_0^1}^{\mathcal{R}}E(u_n)$, $\tilde{u}_n=u_n-\alpha_n g_n$, $R_n=u_{n+1}-\tilde{u}_n$. According to the estimate in \eqref{equation_4_3} and Lemma \ref{lemma_5_2}, we have
    \begin{equation*}
        \|R_n\|_{H_0^1} \le \frac{\alpha_n^2}{2}(C_u+\alpha_n C_g)C_3^2\|g_n\|_{H_0^1}^2 \le \frac{\alpha_n^2}{2}(C_u+\alpha_n C_g)C_3^2 L^2 \delta_n^2. 
    \end{equation*}
    Furthermore,
    \begin{align*}
        \delta_{n+1} &= \min_{\omega\in [-\pi,\pi)} \|u_{n+1} - e^{\mathrm{i}\omega}u^*\|_{H_0^1} \le \|u_{n+1}-u_n^*\|_{H_0^1} = \|R_n + \tilde{u}_n -u_n^*\|_{H_0^1}\\
        &\le \|\tilde{u}_n-u_n^*\|_{H_0^1} + \|R_n\|_{H_0^1} \le \left[\sqrt{P_1(\alpha_n, \delta_n)} + \frac{\alpha_n^2}{2}(C_u+\alpha_n C_g)C_3^2L^2 \delta_n\right]\delta_n.
    \end{align*}
    When there exists a positive number $\theta<1$ satisfying
    \begin{equation}
        \sqrt{P_1(\alpha_n, \delta_n)} + \frac{\alpha_n^2}{2}(C_u+\alpha_n C_g)C_3^2L^2 \delta_n \le \theta < 1,
    \end{equation}
    we can obtain
    \begin{equation}
        \delta_{n+1} \le \theta \delta_n \le \theta^{n+1}\delta_0. 
    \end{equation}
    By induction, we can now conclude that $\{u_n\}_{n=1}^{\infty}$ converges linearly to the ground state. 

    In summary, for this local linear convergence to hold, it is sufficient to ensure that the initial distance $\delta_0 \le C_{\rho}$ and that the iterative step size $\alpha_n$ and distance $\delta_n$ satisfy the inequalities
    \begin{equation}\label{equation_5_45}
        \sqrt{P_1(\alpha_n, \delta_n)} + \frac{\alpha_n^2}{2}(C_u+\alpha_n C_g)C_3^2L^2 \delta_n < 1, \quad\quad P_2(\delta_n) < 0.  
    \end{equation}
    In fact, the conditions for these inequalities to hold can be given explicitly. For example, if we choose constants $C_{\rho}$ and $C_5$ to be 
    \begin{equation*}
        C_{\rho} = \frac{1}{4}C_3^{-3}\|\mathcal{G}_{H_0^1} u^*\|_{H_0^1}^2, \quad\quad C_5 = \frac{K(\lambda_2-\lambda)}{8(1+K)\lambda_2},
    \end{equation*}
    then these two constants satisfy conditions 
    \begin{align*}
        C_{\rho} &\le \frac{1}{2}C_3^{-3}\|\mathcal{G}_{H_0^1} u^*\|_{H_0^1}^2, \\
        C_5 &< \frac{K(\lambda_2-\lambda)}{2(1+K)\lambda_2} - \frac{2\mu C_3 K(\lambda_2-\lambda)}{8(1+K)\mu\lambda_2 C_3} \le \frac{K(\lambda_2-\lambda)}{2(1+K)\lambda_2} - 2\mu C_3 \delta_0. 
    \end{align*}
    Furthermore, if the maximum step size $\alpha_{\max}$ and initial distance $\delta_0$ satisfy 
    \begin{align}
        &\delta_0\le 1,\quad\quad \delta_0\le \frac{1}{4}C_3^{-3}\|\mathcal{G}_{H_0^1} u^*\|_{H_0^1}^2,\quad\quad \delta_0\le \frac{K(\lambda_2-\lambda)}{8(1+K)\mu\lambda_2 C_3}, \\
        &\delta_0\le \left(\frac{4C_3^3\lambda+2L_2}{\|\mathcal{G}_{H_0^1} u^*\|_{H_0^1}^2}+\lambda C_3\right)^{-1}C_3^{-2}\frac{K(\lambda_2-\lambda)}{8(1+K)\lambda_2}, \\
        &\delta_0\le \left[\beta C_2 C_3^4\|u^*\|_{L^4}^3 \|u^*\|_{H_0^1}+2\beta \left(C_2 C_3^2\|u^*\|_{L^4}^3+4 C_4^3\right)\right]^{-1}\frac{K(\lambda_2-\lambda)}{16(1+K)\lambda_2}, \\
        &\delta_0\le (C_u+C_g)^{-1}C_3^{-2}L^{-2}\frac{K(\lambda_2-\lambda)}{32(1+K)\lambda_2}, \\
        &\alpha_{\max} \le 1,\quad\quad \alpha_{\max} \le \frac{K(\lambda_2-\lambda)}{32(1+K)\lambda_2 L^2}, 
    \end{align}
    then the polynomial $P_2$ is negative, since 
    \begin{align*}
        P_2(\delta_n) &< \left(\frac{2C_3^3\lambda + L_2}{\|\mathcal{G}_{H_0^1} u^*\|_{H_0^1}^2-2C_3^3C_{\rho}} + \gamma_n^* C_3\right)\delta_0 + C_3^{-2}\left(2\mu C_3 \delta_0 - \frac{K(\lambda_2-\lambda)}{2(1+K)\lambda_2} + C_5\right) \\
        &\le C_3^{-2}\frac{K(\lambda_2-\lambda)}{8(1+K)\lambda_2} - C_3^{-2}\left(\frac{2\mu C_3 K(\lambda_2-\lambda)}{8(1+K)\mu\lambda_2 C_3} - \frac{K(\lambda_2-\lambda)}{2(1+K)\lambda_2} + C_5\right) \le 0.
    \end{align*}
    Finally, we have 
    \begin{equation}\label{equation_5_44}
        \sqrt{P_1(\alpha_n, \delta_n)} + \frac{\alpha_n^2}{2}(C_u+\alpha_n C_g)C_3^2L^2 \delta_n \le \sqrt{P_1(\alpha_n, \delta_n)} + \frac{1}{2}(C_u+C_g)C_3^2L^2 \delta_n \alpha_n. 
    \end{equation}
    The right-hand side of \eqref{equation_5_44} is strictly less than $1$ if and only if 
    \begin{equation}\label{equation_5_40}
        P_1(\alpha_n, \delta_n) < \frac{1}{4}(C_u+C_g)^2C_3^4L^4 \delta_n^2 \alpha_n^2 - (C_u+C_g)C_3^2L^2 \delta_n \alpha_n + 1,
    \end{equation}
    while the left-hand side of \eqref{equation_5_40} is
    \begin{equation}\label{equation_5_41}
        P_1(\alpha_n, \delta_n)\le P_1(\alpha_n, \delta_0) \le L^2\alpha_n^2 - \frac{K(\lambda_2-\lambda)}{16(1+K)\lambda_2}\alpha_n + 1, 
    \end{equation}
    and the right-hand side of \eqref{equation_5_40} is
    \begin{equation}\label{equation_5_42}
        \frac{1}{4}(C_u+C_g)^2C_3^4L^4 \delta_n^2 \alpha_n^2 - (C_u+C_g)C_3^2L^2 \delta_n \alpha_n + 1 > - (C_u+C_g)C_3^2L^2 \delta_0 \alpha_n + 1 \ge - \frac{K(\lambda_2-\lambda)}{32(1+K)\lambda_2}\alpha_n + 1. 
    \end{equation}
    Combining \eqref{equation_5_41} and \eqref{equation_5_42}, if the following inequality is satisfied, \eqref{equation_5_40} hold. 
    \begin{equation}\label{equation_5_46}
        L^2\alpha_n^2 - \frac{K(\lambda_2-\lambda)}{16(1+K)\lambda_2}\alpha_n + 1 \le - \frac{K(\lambda_2-\lambda)}{32(1+K)\lambda_2}\alpha_n + 1.
    \end{equation}
    \eqref{equation_5_46} is true when $0<\alpha_n\le \alpha_{\max}\le \frac{K(\lambda_2-\lambda)}{32(1+K)\lambda_2 L^2}$. In conclusion, with such parameter choices, inequalities \eqref{equation_5_45} are satisfied, implying that $\{u_n\}_{n=1}^{\infty}$ converges linearly.
\end{proof}

Finally, we prove Theorem \ref{theorem_3_7}.

\begin{proof}[Proof of Theorem \ref{theorem_3_7}]
    Define 
    \begin{align*}
        &\tilde{\omega}_n = \argmin_{\omega\in[-\pi,\pi)} \|u_n - e^{\mathrm{i}\omega}u^*\|_{a_{u_n}},\qquad \tilde{u}_n^*=e^{\mathrm{i}\tilde{\omega}_n}u^*,\qquad \tilde{e}_n = \tilde{u}_n^* - u_n,\qquad \tilde{\delta}_n = \|\tilde{e}_n\|_{a_{u_n}},\\
        &\omega_n = \argmin_{\omega\in[-\pi,\pi)} \|u_n - e^{\mathrm{i}\omega}u^*\|_{a_{u^*}},\qquad u_n^*=e^{\mathrm{i}\omega_n}u^*,\qquad e_n = u_n^* - u_n,\qquad \delta_n = \|e_n\|_{a_{u^*}}. 
    \end{align*}
    Since $H_0^1$ norm and $a_{u^*}$ norm are equivalent, we only need to show that $\{\delta_n\}_{n=0}^{\infty}$ converges to $0$ at an exponential rate. 

    Similar to the proof of Theorem \ref{theorem_3_5}, we prove it by induction, and assume that $\delta_n\le \tilde{\theta}^n\delta_0$ for some positive constant $\tilde{\theta}<1$, whose value will be given in the proof below. This assumption is trivial when $n=0$. According to~\cite{henning2025convergence}, $\{u_n\}_{n=0}^{\infty}$ is bounded in $H_0^1$ norm. By the norm equivalence in Lemma \ref{lemma_2_2}, for any $n\ge 0$, we have
    \begin{align*}
        \|u_n\|_{a_{u_n}}^2 &\le \left(1 + \frac{1+K}{K}\beta C_1^4 \|u_n\|_{H_0^1}^2\right)\left(1 + \Omega M C_3 + C_3^2\|V\|_{L^\infty}\right) \|u_n\|_{H_0^1}^2, \\
        \|u_n\|_{a_{u^*}}^2 &\le \left(1 + \frac{1+K}{K}\beta C_1^4 \|u^*\|_{H_0^1}^2\right)\left(1 + \Omega M C_3 + C_3^2\|V\|_{L^\infty}\right) \|u_n\|_{H_0^1}^2, 
    \end{align*}
    which implies that $u_n$ is bounded in both $a_{u_n}$ and $a_{u^*}$ norms. Similarly, $u^*$ is also bounded in these norms. We denote their uniform upper bound by $\tilde{C}_u$, i.e.,
    \begin{equation*}
        \sup_{n\ge 0}(\|u_n\|_{a_{u_n}}, \|u_n\|_{a_{u^*}}, \|u^*\|_{a_{u_n}}, \|u^*\|_{a_{u^*}}) \le \tilde{C}_u. 
    \end{equation*}
    Recall that $|u^*| = |\tilde{u}^*_n|$.  Now, for any $v\in H_0^1(\mathcal{D})$, we have
    \begin{align*}
        \left|\|v\|_{a_{u_n}}^2 - \|v\|_{a_{u^*}}^2\right| &= \beta\left|\int_{\mathcal{D}} \left(|u_n|^2-|\tilde{u}_n^*|^2\right)|v|^2 dr\right| \\
        &\le \beta\|u_n-\tilde{u}_n^*\|_{L^4} \|u_n+\tilde{u}_n^*\|_{L^4} \|v\|_{L^4}^2 \\
        &\le 2\beta C_1^4 \left(\frac{1+K}{K}\right)^2 \tilde{C}_u\tilde{\delta}_n \|v\|_{a_{u_n}}^2. 
    \end{align*}
    Therefore,
    \begin{equation*}
        \|v\|_{a_{u^*}}^2 \le \left[1+2\beta C_1^4\left(\frac{1+K}{K}\right)^2 \tilde{C}_u\tilde{\delta}_n\right] \|v\|_{a_{u_n}}^2 \le \left[1+\beta C_1^4\left(\frac{1+K}{K}\right)^2 \tilde{C}_u\tilde{\delta}_n\right]^2 \|v\|_{a_{u_n}}^2. 
    \end{equation*}
    Denote $\tilde{K} = \beta C_1^4\left(\frac{1+K}{K}\right)^2 \tilde{C}_u$. Taking $v = \tilde{e}_n$, we get
    \begin{equation*}
        \delta_n = \min_{\omega} \|u_n-e^{\mathrm{i}\omega}u^*\|_{a_{u^*}} \le \|\tilde{e}_n\|_{a_{u^*}} \le \left(1+\tilde{K}\tilde{\delta}_n\right) \tilde{\delta}_n. 
    \end{equation*}
    Similarly, we have
    \begin{equation*}
        \tilde{\delta}_n \le \left(1+\tilde{K}\delta_n\right) \delta_n. 
    \end{equation*}

    Using the conclusion from Remark \ref{remark_5_1}, similar to the proof of Theorem \ref{theorem_3_5}, it can be shown that there exists a positive number $\theta_1 < 1$ such that for any $n\ge 0$,
    \begin{equation*}
        \tilde{\delta}_{n+1} \le \theta_1 \tilde{\delta}_n. 
    \end{equation*}
    Thus, we obtain
    \begin{equation*}
        \delta_{n+1} \le \left(1+\tilde{K}\tilde{\delta}_{n+1}\right) \tilde{\delta}_{n+1} \le \theta_1 \left(1+\tilde{K}\theta_1 \left(1+\tilde{K}\delta_n\right) \delta_n\right) \left(1+\tilde{K}\delta_n\right) \delta_n. 
    \end{equation*}
    If there exists a positive number $\theta_1 < \tilde{\theta} < 1$ satisfying
    \begin{equation}\label{equation_5_43}
        \left(1+\tilde{K}\theta_1 \left(1+\tilde{K}\delta_n\right) \delta_n\right) \left(1+\tilde{K}\delta_n\right) \le \theta_1^{-1}\tilde{\theta}, 
    \end{equation}
    we have
    \begin{equation*}
        \delta_{n+1} \le \tilde{\theta} \delta_n \le \tilde{\theta}^{n+1} \delta_0. 
    \end{equation*}
    By induction, $\{\delta_n\}_{n=0}^{\infty}$ decays exponentially to $0$ when $n\to\infty$. Because of the hypothesis of induction, $\delta_n\le \tilde{\theta}^n \delta_0$, thus the polynomial constraint \eqref{equation_5_43} holds when $\delta_0$ is sufficiently small.
\end{proof}

\section{Numerical Experiments}\label{section_6}

We conclude this paper with numerical experiments conducted in an environment equipped with an AMD Ryzen 5 4600H (CPU) and an NVIDIA GeForce GTX 1660 Ti (GPU). In the numerical experiments of this section, we consider the optimization problem of the Gross-Pitaevskii energy functional \eqref{equation_2_6} on the 2D square domain $\mathcal{D}=[-6,6]^2$ with parameters as in~\cite{henning2025convergence}. Specifically, we use an anisotropic harmonic trapping potential
\begin{equation*}
    V(x,y)=\frac{1}{2}\left((0.9x)^2 + (1.2y)^2\right),
\end{equation*}
a particle interaction parameter $\beta=100$, and an angular velocity $\Omega=1.2$. For spatial discretization, we use $\mathbb{P}^2$-Lagrange finite elements on a uniform mesh with $(2^8-1)^2$ degrees of freedom. Using the $a_u$ scheme and setting the stopping criterion to a residual of $5\times 10^{-14}$, we compute the minimum energy, corresponding eigenvalue, and the residual of the Gross-Pitaevskii eigenvalue equation as
\begin{align*}
    &E(u^*)\approx 1.64353578,\quad\quad \lambda\approx 4.44781457,\\
    &\left\|-\Delta u^* + Vu^* + \beta|u^*|^2u^* - \Omega L_z u^* - \lambda u^*\right\|_{L^{\infty}}=5.2165\times 10^{-14}. 
\end{align*}
Due to the higher accuracy of $\mathbb{P}^2$ finite elements, these values are smaller than those reported in~\cite{henning2025convergence}. The shape of the ground state solution is shown in Figure \ref{figure_1}, where the left panel displays the density $|u^*|^2$ and the right panel shows the phase $\arg(u^*)$. This result will serve as the benchmark ground state for subsequent experiments.

\begin{figure}[!htp]
    \centering
    \subfigure[density $|u^*|^2$]{
        \includegraphics[width=0.45\textwidth]{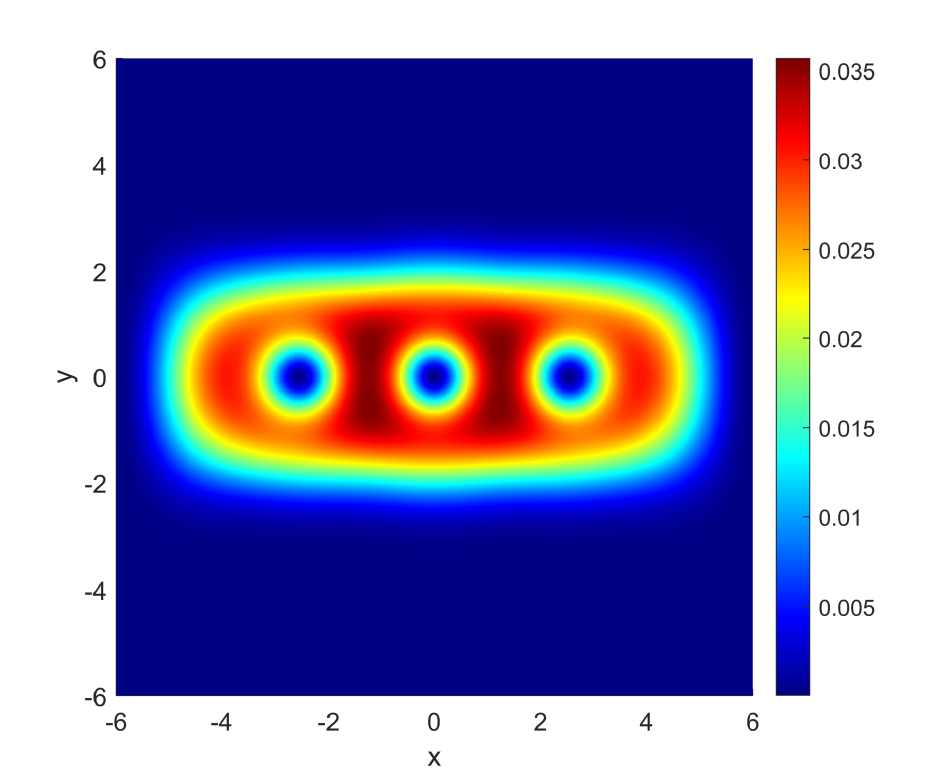}
    }
    \subfigure[phase $\arg(u^*)$]{
        \includegraphics[width=0.45\textwidth]{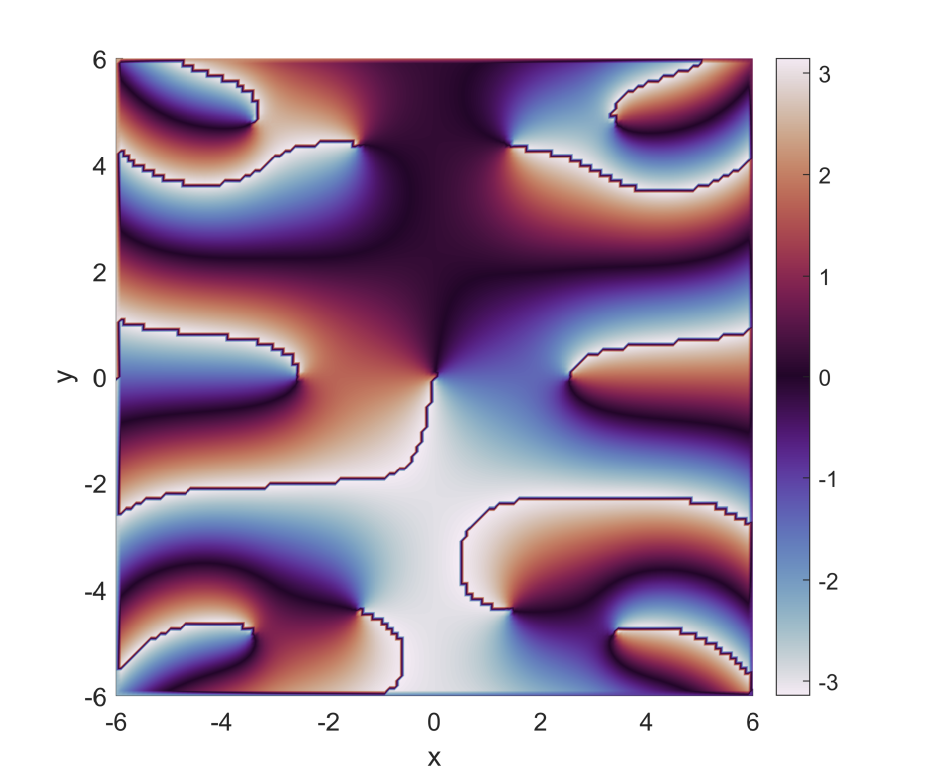}
    }
    \caption{Density and phase of the ground state $u^*$. }
    \label{figure_1}
\end{figure}
\begin{figure}[!htp]
    \centering
    \subfigure[$H_0^1$ scheme]{
        \includegraphics[width=0.30\textwidth]{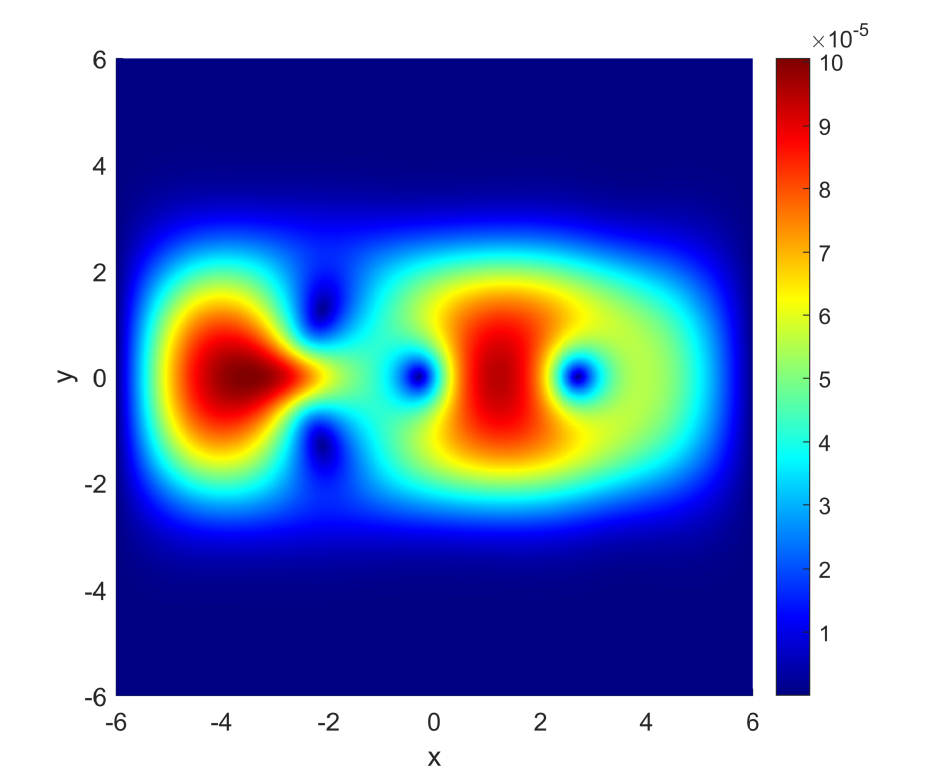}
    }
    \subfigure[$a_0$ scheme]{
        \includegraphics[width=0.30\textwidth]{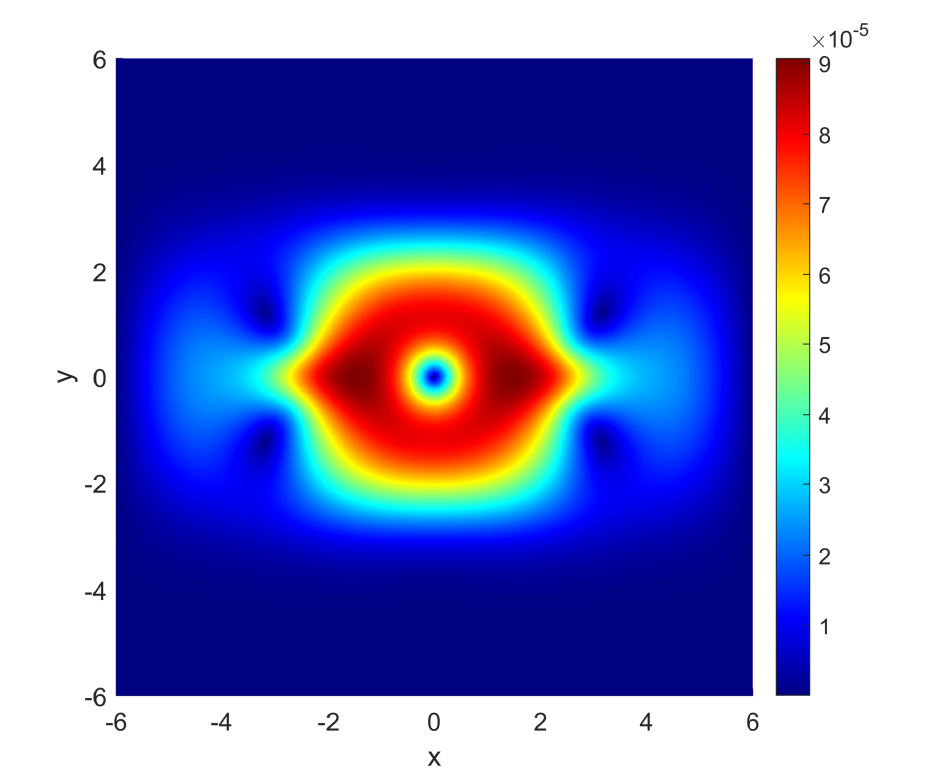}
    }
    \subfigure[$a_u$ scheme]{
        \includegraphics[width=0.30\textwidth]{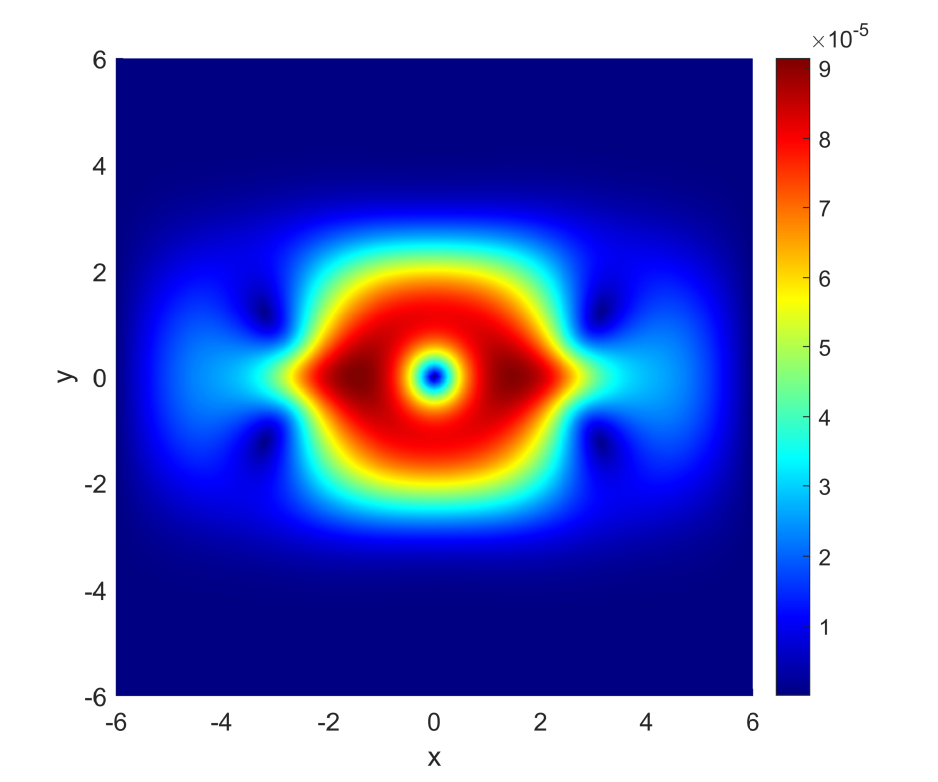}
    }
    \caption{Ground state error between real and calculated ground states.}
    \label{figure_4}
\end{figure}

\subsection{Energy dissipation}
In the first experiment, we explore the energy dissipation rates of three gradient flow schemes ($H_0^1, a_0, a_u$). The initial wave function is set as
\begin{equation*}
    u_0(x,y) = \frac{x+iy}{\sqrt{\pi}}\exp\left(-\frac{x^2+y^2}{2}\right), 
\end{equation*}
and the step size at each iteration is chosen as the optimal step size
\begin{equation*}
    \alpha_n = \argmin_{\alpha} E\left(\frac{u_n-\alpha \nabla_{X}^{\mathcal{R}}E(u_n)}{\|u_n-\alpha \nabla_{X}^{\mathcal{R}}E(u_n)\|_{L^2}}\right). 
\end{equation*}
In the implementation, we use the golden section search~\cite{henning2025convergence}. The stopping criterion is set such that the energy error reached $10^{-9}$. The errors of the calculated ground states with respect to the benchmark ground state mentioned above are shown in Figure \ref{figure_4}. It can be observed that the error of the computed ground states remains within $10^{-5}$. 
\begin{figure}[!htp]
    \centering
    \subfigure[energy error]{
        \includegraphics[width=0.90\textwidth]{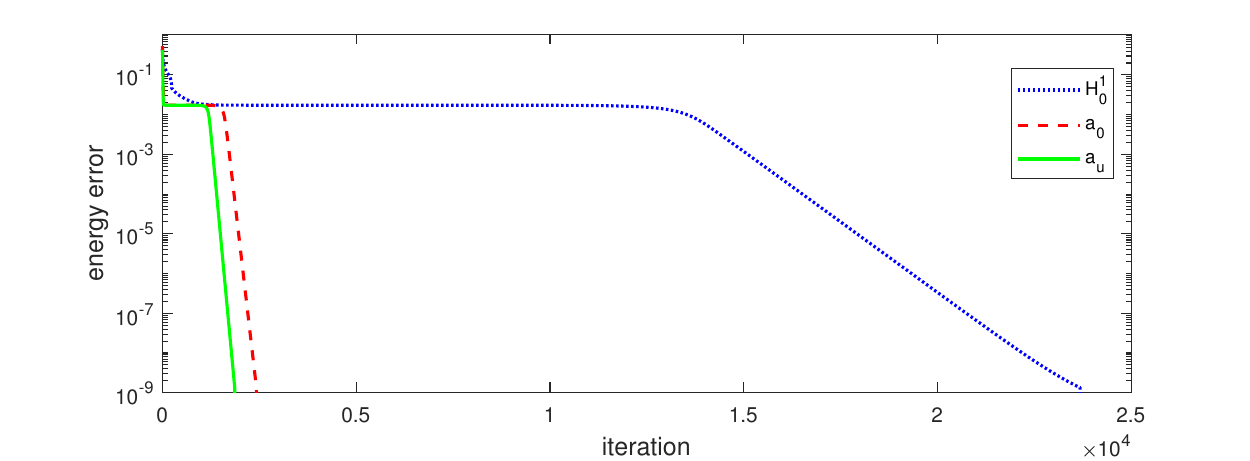}
    }
    \subfigure[optimal step size of $H_0^1$ scheme]{
        \includegraphics[width=0.30\textwidth]{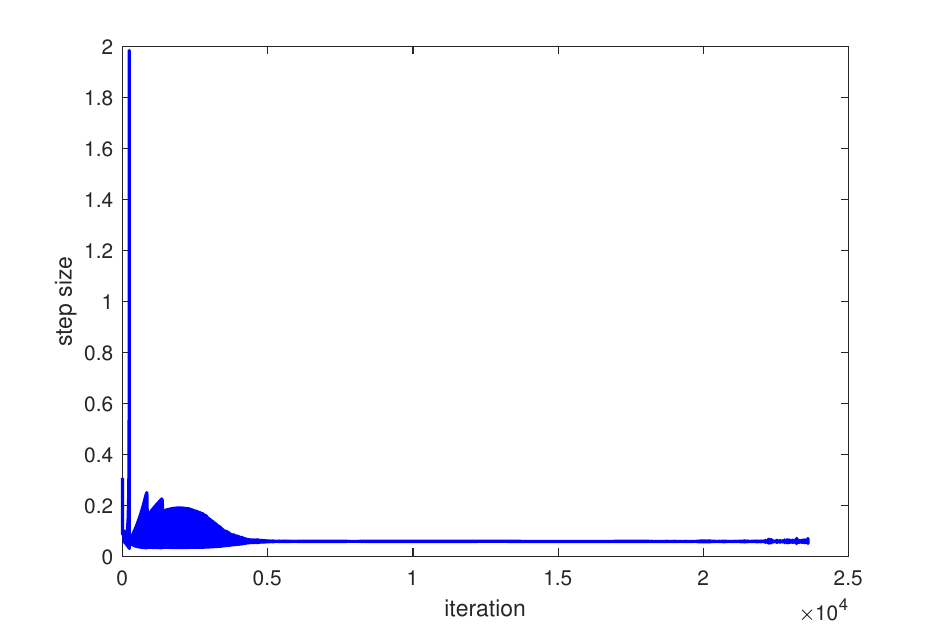}
    }
    \subfigure[optimal step size of $a_0$ scheme]{
        \includegraphics[width=0.30\textwidth]{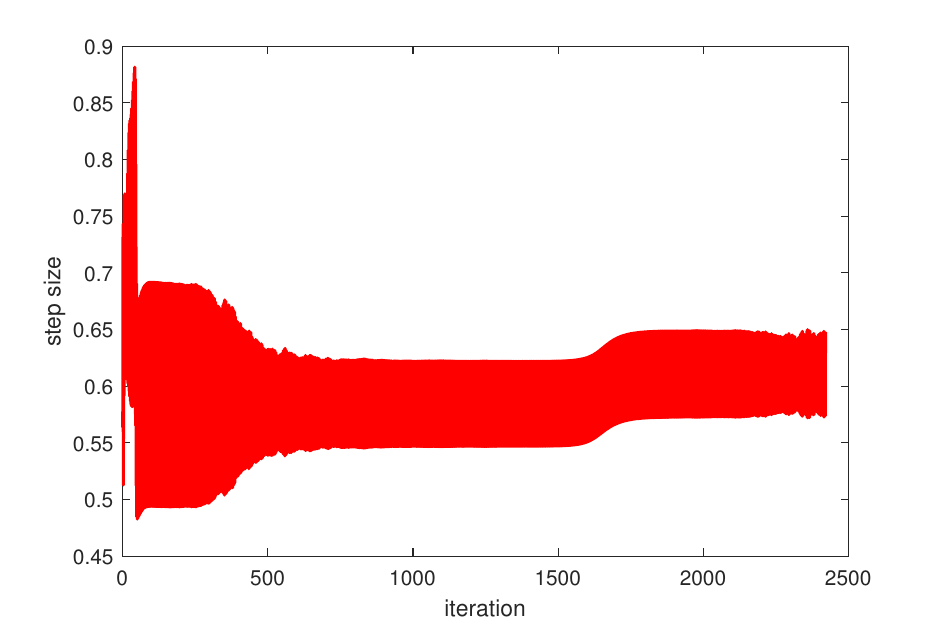}
    }
    \subfigure[optimal step size of $a_u$ scheme]{
        \includegraphics[width=0.30\textwidth]{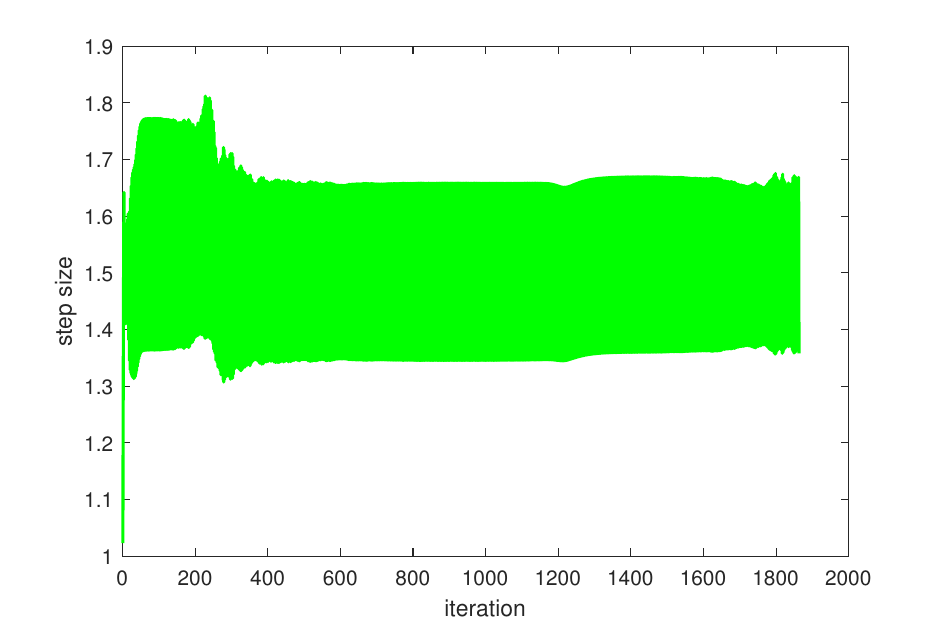}
    }
    \caption{Energy error and optimal step sizes of schemes. }
    \label{figure_2}
\end{figure}
The experimental results for the iterative error are presented in Figure \ref{figure_2}, where the top plot shows the energy error $|E(u_n)-E(u^*)|$ and the bottom plot illustrates the optimal step size $\alpha_n$ for the three schemes. It can be seen that all three schemes have stayed at the same energy level. This phenomenon is attributed to the gradient flow algorithm being trapped at a stationary point, specifically an excited state of the Gross--Pitaevskii energy functional. In the experiment, the $H_0^1$ scheme exhibits a significantly slower energy dissipation rate compared to the $a_0$ and $a_u$ schemes, both in terms of escaping the excited state and achieving linear convergence to the ground state. Furthermore, the optimal step size for the $H_0^1$ scheme during the iteration process is considerably smaller than that of the $a_0$ and $a_u$ schemes. The $a_u$ scheme demonstrate the largest optimal step size and the fastest convergence rate, while the $a_0$ scheme rank in between the two in both aspects. 

\begin{figure}[!htp]
    \centering
    \subfigure[wave function error]{
        \includegraphics[width=0.45\textwidth]{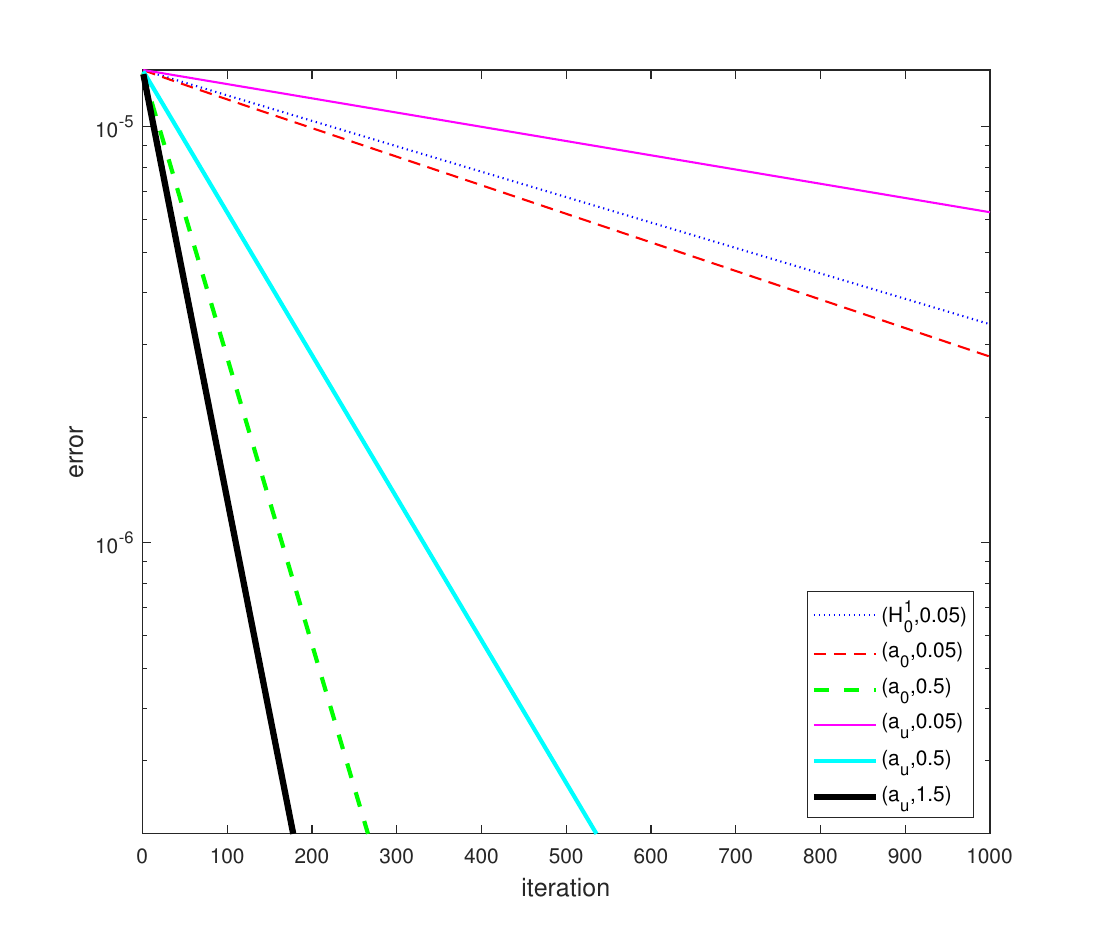}
    }
    \subfigure[error ratio]{
        \includegraphics[width=0.45\textwidth]{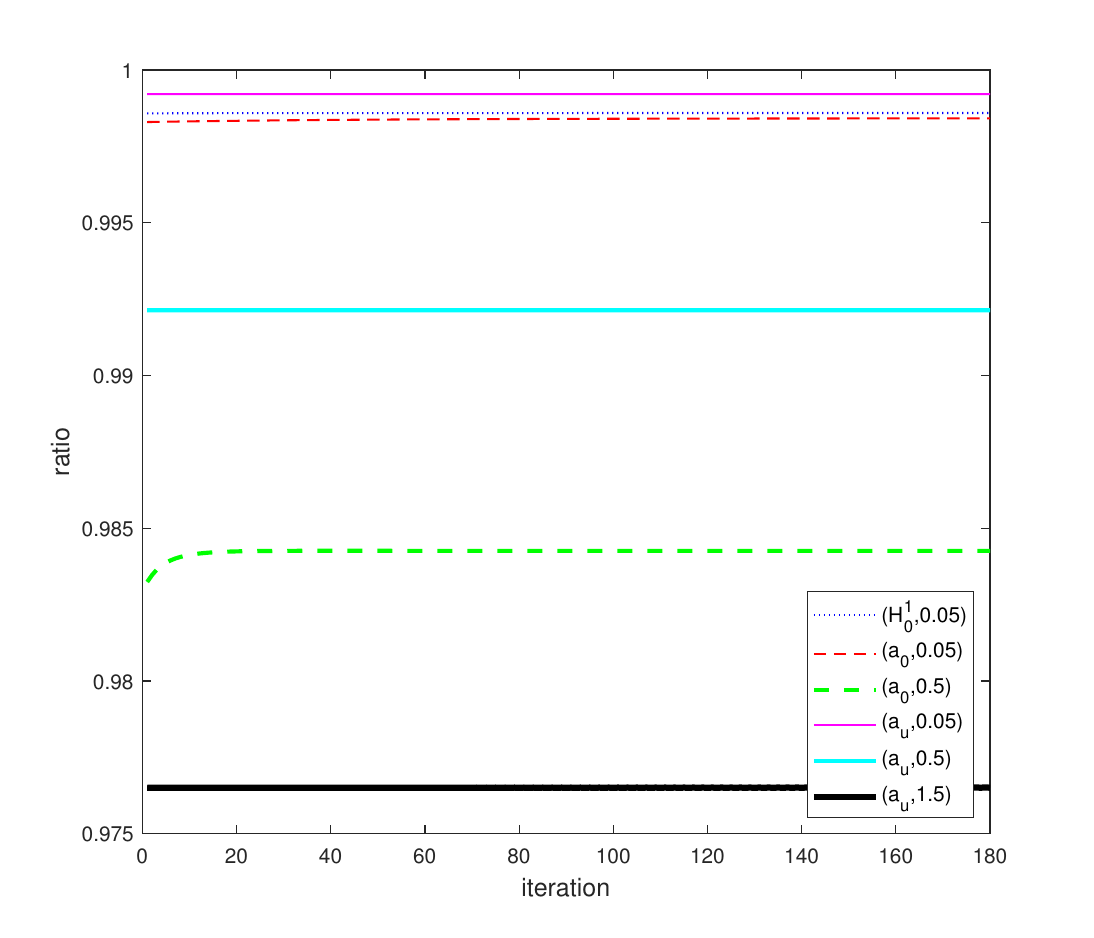}
    }
    \caption{Wave function error and corresponding error ratio.}
    \label{figure_3}
\end{figure}

\subsection{Convergence rates for fixed step sizes}
In the second experiment, we investigate the convergence rates of the solutions for three gradient flow schemes ($H_0^1, a_0, a_u$). To observe the convergence rates of the wave functions for these three gradient flows when close to the ground state, we use an intermediate wave function computed with the $a_u$ scheme from the previous experiment as the initial wave function, which achieved an energy error of $10^{-8}$ for the first time. Considering the step size exploration in the previous experiment, we examine the impact of three fixed step sizes on the convergence rates: $0.05$, $0.5$, and $1.5$. These step sizes approximately correspond to the optimal step sizes for the three gradient flow schemes during iteration. The experimental results are shown in Figure \ref{figure_3}, where the first item in the legend indicates the iteration scheme, and the second item shows the step size. For cases not displayed in Figure \ref{figure_3}, convergence was not achieved. As can be seen in Figure \ref{figure_3}, all three gradient flows demonstrate linear convergence, which corroborates our local linear convergence theorems \ref{theorem_3_5}, \ref{theorem_3_6}, and \ref{theorem_3_7}. We also note that for a step size of $0.05$, the convergence rates of the three gradient flows are very similar. This partially explains the slow energy dissipation rate of the $H_0^1$ scheme in the previous experiment: the optimal step size for the $H_0^1$ scheme is too small. This issue is unresolvable, as the $H_0^1$ scheme fails to converge when larger step size is chosen.

\begin{table}[!ht]
    \centering
    \begin{tabular}{|l|l|l|l|l|l|l|}
    \hline
        Degrees of Freedom & $(2^4-1)^2$ & $(2^5-1)^2$ & $(2^6-1)^2$ & $(2^7-1)^2$ & $(2^8-1)^2$ & $(2^9-1)^2$ \\ \hline
        $H_0^1$ scheme & $0.9990$ & $0.9981$ & $0.9983$ & $0.9983$ & $0.9982$ & $0.9981$  \\ \hline
        $a_0$ scheme & $0.9900$ & $0.9774$ & $0.9800$ & $0.9804$ & $0.9804$ & $0.9762$  \\ \hline
        $a_u$ scheme & $0.9872$ & $0.9730$ & $0.9760$ & $0.9762$ & $0.9761$ & $0.9701$  \\ \hline
    \end{tabular}
    \vspace{20pt}
    \caption[Table 1: ]{impact of different degrees of freedom on the linear convergence rates of the three schemes.}
    \label{table_1}
\end{table}

\subsection{Impact of mesh refinement on convergence rates}
Finally, we investigate the impact of different degrees of freedom on the linear convergence rates of the three gradient flow schemes ($H_0^1, a_0, a_u$). We select an intermediate wave function, which first achieves a wave function error of $10^{-5}$, as the initial wave function. Using the adaptive step size method from Example 1, we compute the average convergence rate over $100$ iterations for each scheme. The results are presented in Table \ref{table_1}. It can be observed that the linear convergence rates of the three schemes are almost unaffected by the degrees of freedom of the mesh. 
	
\section{Summary}

We investigate three Sobolev gradient flow methods for computing the ground states of rotating Bose--Einstein condensates, extending the prior convergence analysis for the non-rotating case to optimization problems involving the Gross--Pitaevskii energy functional with a rotation term. At the theoretical level, we prove the global convergence of the $H_0^1$ and $a_0$ schemes, ensuring that the gradient flow algorithms starting from arbitrary initial values always converge to a stationary point of the energy functional. By using the quotient space and the second-order derivative of the energy functional, we prove the local linear convergence for all three iterative schemes in the neighborhood of a ground state. At the numerical level, our experimental results are highly consistent with the theoretical analysis. Experiment 1 shows the significant advantage of the $a_u$ scheme in convergence speed, Experiment 2 verifies that all three methods achieve linear convergence with appropriate constant step sizes, and Experiment 3 demonstrates the impact of different degrees of freedom of the mesh on the linear convergence rates. 

\subsection*{Funding}

Chen is supported by the National Natural Science Foundation of China (NSFC 12471369 and NSFC 12241101). 

\bibliographystyle{amsxport}
\bibliography{paper}

\end{document}